\documentclass[12pt,twoside]{article}
\usepackage{amssymb}
\usepackage{amsmath}
\usepackage{cases}
\usepackage{txfonts}
\usepackage{amsfonts}
\usepackage{amsmath,amssymb}
\usepackage{multicol}
\usepackage{color}
\usepackage{graphicx}
\usepackage{graphics}

\def\R{\mathbb{R}}

\def\epsilon{\varepsilon}

\newcommand{\be}{\begin{equation}}
\newcommand{\ee}{\end{equation}}
\newcommand{\baa}{\begin{array}}
\newcommand{\eaa}{\end{array}}
\newcommand{\ba}{\begin{eqnarray}}
\newcommand{\ea}{\end{eqnarray}}

\newtheorem{theorem}{Theorem}[section]
\newtheorem{lemma}[theorem]{Lemma}
\newtheorem{corollary}[theorem]{Corollary}

\newtheorem{definition}[theorem]{Definition}
\newtheorem{remark}[theorem]{Remark}

\newtheorem{proposition}[theorem]{Proposition}
\numberwithin{equation}{section}
\newenvironment{proof}[1][Proof]{\noindent\textbf{#1.} }{\hfill $\Box$}
\allowdisplaybreaks

\makeatletter
\begin{document}
\date{}
\title{\bf{ Some new  bistable transition fronts with changing shape}}
\author{Hongjun Guo  \textsuperscript{a}\footnote{Corresponding author. Email: guohj29@tongji.edu.cn}, Kelei Wang \textsuperscript{b}\footnote{Email: wangkelei@whu.edu.cn}\\
		\\
		\footnotesize{ \textsuperscript{a} School of Mathematical Sciences, Key Laboratory of Intelligent Computing and }\\
		\footnotesize{Applications (Ministry of Education), Institute for Advanced Study, Tongji University, Shanghai, China}\\
		\footnotesize{\textsuperscript{b} School of Mathematics and Statistics, Wuhan University, Wuhan, China}}
\maketitle

\begin{abstract}\noindent{}We construct entire solutions of bistable reaction-diffusion equations by mixing finite planar fronts, which form a finite-dimensional manifold. These entire solutions are generalized traveling fronts, that is, transition fronts. We also show their uniqueness and stability. Furthermore, we prove that transition fronts with level sets having finite facets are determined by finite planar fronts and they are in the class of entire solutions constructed by us.
\noindent{}
\vskip 0.1cm
\noindent\textit{Keywords: Entire solutions; Planar fronts; Transition fronts; Uniqueness; Stability.}

\noindent\textit{MSC2020: 35B08; 35C10; 35A09; 35K15.}
\end{abstract}


\section{Introduction}
We are interested in entire solutions of the classical reaction-diffusion equation 
	\be\label{RD}
	\partial_tu-\Delta u=f(u),\quad \text{in}~  \R\times\R^N,
	\ee
	where $N\ge 2$, $0\leq u \leq 1$. The reaction term $f(u)\in C^1([0,1])$ is assumed to be of bistable, that is, there is $\theta\in (0,1)$ such that
	\be\label{F}
	f(0)=f(\theta)=f(1),\ f(u)<0 \hbox{ in $(0,\theta)$},\ f(u)>0 \hbox{ in $(\theta,1)$},\ f^\prime(0)<0 \hbox{ and } f^\prime(1)<0.
	\ee
 It also satisfies the  unbalanced condition,  $\int_0^1 f(u)du\neq 0$. Without loss of generality, we always assume that
	\be\label{signf}
	\int_0^1 f(u)du> 0.
	\ee
	A typical model is $f(u)=u(u-\theta)(1-u)$.
	Such an equation arises in various mathematical modeling in biology, ecology, material sciences and so on.

We will construct  entire solutions $u$ which are monotone in a direction $e\in \mathbb{S}^{N-1}$, that is,
	\be\label{DMonotone}
	e\cdot \nabla u<0.
	\ee
 A typical solution satisfying \eqref{DMonotone} is the planar traveling front $g(x\cdot e-c_f t)$, where $e\in\mathbb{S}^{N-1}$, $g$ and $c_f$ satisfy the one-dimensional equation
	\begin{eqnarray}\label{PF}
		\left\{\begin{array}{lll}
			-c_fg^\prime-g^{\prime\prime}-f(g)=0,\\
			g(-\infty)=1,\ g(+\infty)=0.
		\end{array}
		\right.
	\end{eqnarray} 
	The existence of $(g,c_f)$ is given by Fife and McLeod \cite{FM}. The function $g$ is called profile and the constant $c_f$ is called speed. The speed $c_f$ has the same sign with $\int_0^1 f(u)du$. So $c_f>0$ under our assumption \eqref{signf}. The profile $g$ and the speed $c_f$ are uniquely determined by the reaction term $f$ in the sense that  any other traveling front $(g,c)$ satisfies $g(\cdot)=g(\cdot+\tau)$ for some $\tau\in\R$ and $c=c_f$.   Moreover, $g$ is a decreasing function, see \cite{FM}.  In the sequel, we fix $g$ by setting $g(0)=1/2$. Level sets of the planar front $g(x\cdot e- c_f t)$ are hyperplanes $\{x\in\R^N;\ x\cdot e -c_f t=g^{-1}(\lambda)\}$ (for any $\lambda\in (0,1)$). 
 
 Non-planar fronts of \eqref{RD} under \eqref{signf} are also known to exist in dimensions $N\ge 2$. In \cite{HMR}, Hamel, Monneau and Roquejoffre constructed conical shaped front $u(t,x)=\phi(|x^\prime|,x_N-c t)$, where $x^\prime=(x_1,\cdots,x_{N-1})$ and $c$ can be any constant larger than $c_f$. The conical front satisfies
	\begin{eqnarray*}
		\left\{\begin{array}{lll}
			\phi(r,y)\rightarrow 1 \hbox{ (resp. $0$) uniformly as } y-\psi(r)\rightarrow -\infty \hbox{ (respectively $+\infty$)},\\
			\psi^\prime(+\infty)=\frac{\sqrt{c^2-c_f^2}}{c_f},
		\end{array}
		\right.
	\end{eqnarray*}
	for some $C^1$ function $\psi: [0,+\infty)\rightarrow \R$. In particular, level sets of the conical front are not hyperplanes. For $N=2$, the existence of conical fronts is also proved by Ninomiya and Taniguchi \cite{NT1,NT2} by the method of  sub- and supersolutions. In this case, level sets are asymptotic to rays, that is,
	\[\psi(x_1)-\frac{\sqrt{c^2-c_f^2}}{c_f}|x_1|\rightarrow 0, \hbox{ as } |x_1|\rightarrow +\infty.\]
	In particular, level sets are V-shaped. So the conical front in dimension $N=2$ is also called V-shaped front. If one lets $\alpha\in (0,\pi/2)$ such that $\sin\alpha=c_f/c$, the V-shaped front $V(t,x)=\phi(|x_1|,x_2-ct)$ satisfies
	\be\label{VF}
	\lim_{R\rightarrow +\infty} \sup_{x_1^2+(x_2-ct)^2>R^2} \left|V(t,x)-g(-|x_1|\cos\alpha+x_2\sin\alpha-c_f t)\right|=0. 
	\ee
	It means that the V-shaped front  converges to planar fronts along asymptotic lines $x_2=|x_1|\cot\alpha+c_f t$. We refer to this property as ``asymptotic one-dimensional symmetry". For dimensions $N\ge 3$, pyramidal shaped fronts of \eqref{RD} are constructed by Taniguchi \cite{T1} for $N=3$ and Kurokawa and Taniguchi \cite{KT} for $N\ge 4$. Unlike conical fronts, the pyramidal fronts are non-axisymmetric in $x$. For pyramidal fronts, their level sets are asymptotic to hyperplanes. For instance, \eqref{RD} admits a unique pyramidal front $U$ satisfying
	\[\max_{1\le i\le n}\left\{g(x^\prime\cdot e_i^\prime \cos\alpha +x_N\sin\alpha -c_f t)\right\}< U(t,x)<1, \hbox{ for $(t,x)\in\R\times\R^N$},\]
where $x=(x^\prime,x_N)$ with $x^\prime\in\R^{N-1}$	and
	\be\label{PyF}
	\lim_{\gamma\rightarrow +\infty} \sup_{(t,x)\in D(\gamma)} \Big|U(t,x) -\max_{1\le i\le n} \left\{g(x^\prime\cdot e_i^\prime \cos\alpha +x_N\sin\alpha -c_f t)\right\}\Big|=0,
	\ee
	where $\alpha=\arcsin (c_f/c)$ for some $c>c_f$, $e_i^\prime\in\mathbb{S}^{N-2}$, $D(\gamma)=\{(t,x)\in \R\times\R^{N}; dist((t,x),\Gamma)>\gamma\}$ and $\Gamma$ is the set of ridges of the pyramid 
	\[
	\left\{(t,x)\in\R\times\R^N; \min_{1\le i\le n} \{x^\prime\cdot e_i^\prime \cos\alpha +x_N\sin\alpha-c_f t\}\ge 0\right\}.
	\]
	See \cite{KT,T1,T2} for details. The notion ``ridge" will be clarified later. Level sets of the pyramidal front  look like the surface of the pyramid. It satisfies the asymptotic one-dimensional symmetry property, that is, along one facet of the surface of the pyramid and away from other facets, the pyramidal front $U$  converges to a planar front. The pyramidal front could be generalized such that every planar front contains a shift, see \cite{T1}. 
 
 These fronts recalled here all satisfy the monotone property \eqref{DMonotone}  and all of them are traveling fronts, that is, they have the form $\phi(x^\prime,x_N-ct)$. So level sets of these fronts are invariant in time. More results about the symmetry, stability, uniqueness and other qualitative properties of axisymmetric and non-axisymmetric traveling fronts can be found in \cite{HM,HMR2,RR,T2,T3}. 
	
However, there are plenty of entire solutions other  than traveling fronts, even if under the monotone condition \eqref{DMonotone}. In \cite{H}, Hamel constructed an entire solution with varying level sets in dimension $N=2$, which can also be viewed as solutions in dimensions $N\ge 3$ by a trivial extension.  This entire solution is constructed by rotating the V-shaped front such that one ray of the level set is horizontal and then attaching the reflected V-shaped front to it. This entire solution is axially symmetric and propagates as three planar fronts in three directions. But it seems that this method solely can hardly be extended to higher dimensions if we want to mix  more than three planar fronts. Instead, by combining  Hamel's idea of mixing planar fronts   with the construction of V-shaped fronts and pyramidal fronts and utilizing the  asymptotic one-dimensional symmetry of these solutions, we succeed in constructing entire solutions by mixing any finitely many planar fronts so that the  solution still satisfies the monotone condition \eqref{DMonotone}.  Moreover, we show that these solutions can be characterized by having finitely many planar fronts.

\section{Main results}
Let us first introduce some notations. For any $e\in\mathbb{S}^{N-1}$ and $\tau\in\R$,  denote the hyperplane  
	\[P(e,\tau)=\left\{(t,x)\in\R\times\R^N;\ x\cdot e -c_f t+\tau=0\right\},\]
which is also viewed as a hyperplane in $\R^N$ moving in time. Every pair $(e,\tau)$   determines a planar front $g(x\cdot e -c_f t+\tau)$. For any fixed $e\in\mathbb{S}^{N-1}$, all planar fronts $g(x\cdot e-c_ft +\tau)$ for $\tau\in\R$ form a one-dimensional manifold. Now, take $n\ge 2$ pairs $(e_i,\tau_i)$ of $\mathbb{S}^{N-1}\times\R$ such that 
\begin{itemize}
    \item $e_i\neq e_j$ for any $i\neq j$;
    \item there exists an $e_0\in\mathbb{S}^{N-1}$ such that for all $i$, $e_i\cdot e_0>0$.
\end{itemize}
The first condition means that the propagation directions of planar fronts
$g(x\cdot e_i -c_f t +\tau_i)$   and  $g(x\cdot e_j -c_f t +\tau_j)$
are not the same. Let $P_i$ denote the hyperplane $P(e_i,\tau_i)$ for short. Since $e_i\neq e_j$ for $i\neq j$ and $e_i\cdot e_0>0$ for all $i$, $P_i$ and $P_j$ are not parallel. Let $\mathcal{P}(P_1,\cdots,P_n)$ be the  unbounded polytope enclosed by $P_1$, $\cdots$, $P_n$, that is,
	\[\mathcal{P}(P_1,\cdots,P_n):=\{(t,x)\in\R\times\R^N;\ \min_{1\le i\le n}\{x\cdot e_i -c_f t +\tau_i\}\ge 0\}.\]
The unboundedness is ensured by the fact that $e_i\cdot e_0>0$ for all $i$.
We write $\mathcal{P}(P_1,\cdots,P_n)$ by $\mathcal{P}$ for short if there is no confusion. One can also regard $\mathcal{P}$ as a polytope in  $\R^N$ moving in time. By the definition of the polytope $\mathcal{P}$, one can  verify that $\mathcal{P}$ is convex in $\R^{N+1}$ and $\mathcal{P}_t$, its time slice at $t$, is convex in $\R^N$ for every $t$. The boundary  of $\mathcal{P}$ is
	\[\partial \mathcal{P}:=\left\{(t,x)\in\R\times\R^N;\ \min_{1\le i\le n} \{x\cdot e_i -c_f t+\tau_i\}=0\right\}.\]
	The joint part of $P_i$ and $\partial \mathcal{P}$ is called facet of the surface, which is denoted by $\widetilde{P}_i:=P_i\cap \partial\mathcal{P}$. If we regard $\mathcal{P}$ as a polytope in $\R^{N}$ moving in time, then there is $T\in\R$ such that for every $t\le T$, there are $n$ facets on the surface $\partial \mathcal{P}_t$, that is, the joint part of every $P_{i,t}$ and $\partial\mathcal{P}_t$ is not empty. But as $t$ increases, the number of facets may be less than $n$, that is, $\widetilde{P}_{i,t}=\emptyset$ for some $i\in \{1,\cdots,n\}$ if $t\gg 1$. Here we use  $S_t$ to denote the time slice of the set $S$ at time $t$. 
 
 The intersection of every two facets is called a ridge. Let $\mathcal{R}_{ij}=\widetilde{P}_i \cap \widetilde{P}_j$, $i\neq j$, that is, be the ridges. Notice that not every two facets intersect. Let $\mathcal{R}$ be the set of all ridges.

In the following, we say that the functions $u_p$ converge to a function $u_{p_0}$ as $p\rightarrow p_0\in \R^n$ in the sense of the topology $\mathcal{T}$ if for any compact set $K\subset \R^N$, the functions $u_p$, $\nabla u_p$, $\nabla^2 u_p$ and $\partial_t u_p$ converge uniformly in $K$ to $u_{p_0}$, $\nabla u_{p_0}$, $\nabla^2 u_{p_0}$ and $\partial_t u_{p_0}$ as $p\rightarrow p_0$. Now we can present our first main result.
	
	\begin{theorem}[Existence and Uniqueness]\label{Th1}
		Take $n$ vectors $e_i$ ($i=1,\cdots,n$) of  $\mathbb{S}^{N-1}$ such that $e_i\cdot e_0>0$ for a certain $e_0\in\mathbb{S}^{N-1}$ and $e_i\neq e_j$ for $i\neq j$. Take $n$ constants $\tau_i$ ($i=1,\cdots,n$). Then, $\{(e_i,\tau_i)\}_{i=1,\cdots,n}$ uniquely determines an entire solution $U(t,x;e_1,\cdots,e_n,\tau_1,\cdots,\tau_n)$ of \eqref{RD} satisfying 
		\be\label{compareUu}
		\max_{1\le i\le n}\left\{g(x\cdot e_i -c_f t+\tau_i)\right\}<U_{e_i,\tau_i}(t,x)<1, \hbox{ for all $(t,x)\in\R\times\R^N$},
		\ee
		and
		\be\label{AsyES}
		\left|U_{e_i,\tau_i}(t,x)-\max_{1\le i\le n}\{g(x\cdot e_i -c_f t+\tau_i)\}\right|\rightarrow 0,
		\ee
		uniformly as $d((t,x),\mathcal{R})\rightarrow +\infty$.  
  
  Moreover, $U(t,x;e_1,\cdots,e_n,\tau_1,\cdots,\tau_n)$ satisfies the following two properties.
  \begin{enumerate}
      \item For any $(t_0,x_0)\in\R\times\R^N$ such that $\min_{i=\{1,\cdots,n\}}\{x_0\cdot e_i -c_f t_0\}\ge 0$,
      \[U_{e_i,\tau_i}(t-t_0,x-x_0)\ge U_{e_i,\tau_i}(t,x),\]
where the strict inequality holds if $\min_{i=\{1,\cdots,n\}}\{x_0\cdot e_i -c_f t_0\}> 0$.
      \item For fixed $(e_1,\cdots,e_n)\in\mathbb{S}^{N-1}\times\cdots\times\mathbb{S}^{N-1}$,  $U(t,x;e_1,\cdots,e_n,\tau_1,\cdots,\tau_n)$ are decreasing in $\tau_i\in\R$ for every $i\in\{1,\cdots,n\}$ and depend continuously on $(\tau_1,\cdots,\tau_n)\in\R^n$ in the sense of $\mathcal{T}$.
  \end{enumerate}
	\end{theorem}
	
	\begin{remark}
 The property 1 implies that $e_0\cdot \nabla U_{e_i,\tau_i}(t,x)<0$.
		By the uniqueness of this solution, the above result gives the following examples.
  \begin{itemize}
      \item 
 In dimension $N=2$, if we take $e_1=(\cos\alpha,\sin\alpha)$, $e_2=(-\cos\alpha,\sin\alpha)$ with $\alpha\in (0,\pi/2)$, and $\tau_1=\tau_2=0$, then $U_{e_i,\tau_i}$ is the V-shaped front satisfying \eqref{VF} with $c=c_f/\sin\alpha$. 
  
  \item Still in dimension $N=2$, if we take $e_1=(\cos\alpha,\sin\alpha)$, $e_2=(-\cos\alpha,\sin\alpha)$ with $\alpha\in (0,\pi/2)$, $e_3=(0,1)$ and $\tau_1=\tau_2=\tau_3=0$, then $U_{e_i,\tau_i}$ is the entire solution with varying level sets in \cite{H}.
  
 \item  In dimensions $N\ge 3$, if we take $e_i=(e_i^\prime\cos\alpha,\sin\alpha)$ with $e_i^\prime\in \mathbb{S}^{N-2}$ and $\alpha\in (0,\pi/2)$ and $\tau_i=0$ for $i=1,\cdots,n$, then $U_{e_i,\tau_i}$ is the pyramidal front satisfying \eqref{PyF} where $c=c_f/\sin\alpha$.  
 
 \item In Theorem~\ref{Th1}, if $e_i$ satisfy $e_i\cdot e_0=\text{constant}$ for all $i\in\{1,\cdots,n\}$, then $U_{e_i,\tau_i}$ is a generalized pyramidal front.
   \end{itemize}
	\end{remark}


Our next result concerns the stability of the entire solutions  constructed in Theorem~\ref{Th1}. This is inspired by some known stability results, e.g. the stability of V-shaped fronts (see \cite{NT1,NT2}). In fact, in dimension $N=2$, the stability of $U_{e_i,\tau_i}$ is equivalent to that of V-shaped fronts. This is because although $U_{e_i,\tau_i}$ is constructed by mixing more than two planar fronts, it can be viewed as a V-shaped front with perturbations decaying as $|x|\rightarrow +\infty$ at any time slice. For example, if we take $e_1=(\cos\alpha,\sin\alpha)$, $e_2=(-\cos\alpha,\sin\alpha)$ with $\alpha\in (0,\pi/2)$, $e_3=(0,1)$ and $\tau_1=\tau_2=\tau_3=0$, the entire solution $U_{e_i,\tau_i}$ satisfies the same asymptotic behaviour as $V$	at any $t_0\in\R$, that is,
	\[\lim_{R\rightarrow +\infty} \sup_{x_1^2+x_2^2>R^2} \left|U_{e_i,\tau_i}(t_0,x)-g(-|x_1|\cos\alpha+x_2\sin\alpha-c_f t_0)\right|=0.\]
Hence one can regard $U_{e_i,\tau_i}(t_0,x)-V(t_0,x)$ as a perturbation for the V-shaped front $V(t,x)$ at $t=t_0$. But for dimensions $N\ge 3$,  the stability of $U_{e_i,\tau_i}(t,x)$ is a more general statement than that of pyramidal fronts. This is because in these dimensions, for pyramidal fronts, one should have $e_i\cdot e_0\equiv \text{constant}>0$ for all $i$ and a certain $e_0\in\mathbb{S}^{N-1}$, but if we  take for example $e_i=(e^\prime_i\cos\theta_i,\sin\theta_i)$ and $\tau_i=0$, $i=1,\cdots,4$, where $e^\prime_1=-e^\prime_2\neq e^\prime_3=-e^\prime_4$, then the entire solution $U_{e_i,\tau_i}(t,x)$  is neither a pyramidal front nor asymptotic to a pyramidal front (because its shapes of level sets are changing in $t$). Finally, as far as we know, the stability of pyramidal fronts in dimensions $N\ge 4$ was unsolved before this work, see \cite{KT}.

\begin{theorem}[Stability]\label{Th2}
Suppose $u$ is a solution to the Cauchy problem for \eqref{RD}  with initial value $u(0,x)=u_0(x)$ satisfying
\[
		\left|u_0(x)-\max_{1\le i\le n}\left\{g(x\cdot e_i +\tau_i)\right\}\right|\rightarrow 0
\]
uniformly as $d(x,\mathcal{R}_{0})\rightarrow +\infty$.
Then
		\[\lim_{t\rightarrow +\infty} \sup_{x\in\R^N} \left|u(t,x)-U_{e_i,\tau_i}(t,x)\right|=0.\]
	\end{theorem}
	
	
	Although we have established the existence of many entire solutions satisfying \eqref{DMonotone} in Theorem~\ref{Th1},  there are more entire solutions satisfying \eqref{DMonotone}  than what we constructed. The conical front in $N\ge 3$ is such an example. 
 Then a natural problem is whether entire solutions satisfying \eqref{DMonotone} can be classified, which is however far beyond the scope of this paper. Instead we will restrict our attention only to the characterization of  transition fronts whose level sets are asymptotic to finitely many hyperplanes at $t=-\infty$. We show that they are exactly the entire solutions constructed in Theorem~\ref{Th1}.
	
For this purpose, we need to recall some notations and the definition of transition front, which were introduced by Berestycki and Hamel in \cite{BH1,BH2}. For any two subsets $A$ and $B$ of $\mathbb{R}^N$ and for $x\in\R^N$, we set
	\begin{equation*}
		d(A,B)=\inf\left\{|x-y|;\ (x,y)\in A\times B\right\}
	\end{equation*}
	and $d(x,A)=d(\{x\},A)$, where $|\cdot|$ is the Euclidean norm in $\mathbb{R}^N$. We will also use $d_H$ to denote the Hausdorff distance  between two sets.
 
 Consider two families $(\Omega_t^-)_{t\in \mathbb{R}}$ and $(\Omega_t^+)_{t\in \mathbb{R}}$ of open nonempty subsets of $\mathbb{R}^N$ satisfying
	\begin{eqnarray}\label{eq1.3}
		\forall t\in \mathbb{R},\ \ \left\{
		\begin{aligned}
			&\Omega_t^-\cap \Omega_t^+=\emptyset,\\
			&\partial \Omega_t^-=\partial \Omega_t^+=:\Gamma_t,\\
			&\Omega_t^-\cup \Gamma_t \cup \Omega_t^+=\mathbb{R}^N,\\
			&\sup\{d(x,\Gamma_t);\ \ x\in \Omega_t^+\}=\sup\left\{d(x,\Gamma_t);\ \ x\in \Omega_t^-\right\}=+\infty
		\end{aligned}
		\right.
	\end{eqnarray}
	and
	\begin{eqnarray}\label{eq1.4}
		\left\{
		\begin{aligned}
			&\inf\left\{\sup\left\{d(y,\Gamma_t);\ y\in \Omega_t^+,\ |y-x|\leq r\right\};\ \ t\in \mathbb{R},\ \ x\in \Gamma_t\right\}\rightarrow +\infty\\
			&\inf\left\{\sup\left\{d(y,\Gamma_t);\ y\in \Omega_t^-,\ |y-x|\leq r\right\};\ \ t\in \mathbb{R},\ \ x\in \Gamma_t\right\}\rightarrow +\infty
		\end{aligned}
		\right.
		\text{ as}\ \ r\rightarrow +\infty.
	\end{eqnarray}
	Notice that the condition~\eqref{eq1.3} implies   that for every $t\in \mathbb{R}$, the interface $\Gamma_t$ is not empty  and $\Omega_t^{\pm}$ can not be bounded or contained in a cylinder neither. The sets $\Gamma_t$ are assumed to be made of a finite number of graphs, that is, there is an integer $n\geq 1$ such that, for each $t\in \mathbb{R}$, there are $n$ open subsets $\omega_{i,t}\subset \mathbb{R}^{N-1}$(for $1\leq i\leq n$), $n$ continuous maps $\psi_i: \omega_{i,t}\rightarrow \mathbb{R}$ and $n$ rotations $R_{i,t}$ of $\mathbb{R}^N$, such that
	\begin{equation}\label{eq1.6}
		\Gamma_t \subset \bigcup_{1\leq i\leq n} R_{i,t}\left(\{x\in \mathbb{R}^N; \ \ x^\prime\in \omega_{i,t},\ \ x_N=\psi_i(x^\prime)\}\right).
	\end{equation}
 Further properties of $\Gamma_t$ can be found in \cite{GH,H}.
	
	\begin{definition}\label{def1.1}{\rm{\cite{BH1,BH2}}}
		For~\eqref{RD}, a transition front connecting $0$ and $1$ is a classical solution $u:\mathbb{R}\times\mathbb{R}^N \rightarrow (0,1)$ for which there exist some sets $(\Omega_t^{\pm})_{t\in \mathbb{R}}$ and $(\Gamma_t)_{t\in \mathbb{R}}$ satisfying~\eqref{eq1.3},~\eqref{eq1.4},~\eqref{eq1.6} and for every $\varepsilon>0$, there exists an $M_{\varepsilon}>0$ such that
		\begin{eqnarray}\label{eq1.7}
			\left\{\begin{aligned}
				&\forall t\in \mathbb{R},\ \ \forall x\in \Omega_t^+, \ \ \left(d(x,\Gamma_t)\geq M_{\varepsilon}\right)\Rightarrow \left(u(t,x)\geq 1-\varepsilon\right)\!,\\
				&\forall t\in \mathbb{R},\ \ \forall x\in \Omega_t^-, \ \ \left(d(x,\Gamma_t)\geq M_{\varepsilon}\right)\Rightarrow \left(u(t,x)\leq \varepsilon\right)\!.
			\end{aligned}
			\right.
		\end{eqnarray}
		Furthermore, $u$ is said to have a global mean speed $\gamma$ $(\geq 0)$ if
		\begin{equation*}
			\frac{d(\Gamma_t,\Gamma_s)}{|t-s|}\rightarrow \gamma \ \ \text{as}\ \ |t-s|\rightarrow +\infty.
		\end{equation*}
	\end{definition}
	
	\begin{remark}\label{rmk 2.5}
	By Definition~\ref{def1.1}, one can replace $\Gamma_t$ by $\widetilde{\Gamma}_t$ if $d_{H}(\Gamma_t,\widetilde{\Gamma}_t)<+\infty$ (and replace $\Omega_t^{\pm}$ by $\widetilde{\Omega}^{\pm}_t$, respectively, where $\widetilde{\Omega}^{\pm}_t$ are the two disjoint domains separated by $\widetilde{\Gamma}_t$). 
On the other hand, one can easily check that $U_{e_i,\tau_i}$ in Theorem~\ref{Th1} is a transition front connecting $0$ and $1$ with sets
\be\label{definition of Gamma}
\Gamma_t:=\left\{x\in\R^N;\ \min_{1\le i\le n} \{x\cdot e_i -c_f t+\tau_i\}=0\right\},
\ee
		and
		\be\label{definition of Omega}
  \Omega_t^{\pm}:=\left\{x\in\R^N; \min_{1\le i\le n} \{x\cdot e_i -c_f t+\tau_i\}\lessgtr 0\right\}.
  \ee
	\end{remark}

	Now we present our last main result.
	
	\begin{theorem}\label{Th3}
		Assume that $u$ is a transition front connecting $0$ and $1$ such that for every $t\ll -1$, there are $n$ vectors $e_i$ (independent of $t$) and constants $\xi_t^i$ ($i=1,2,\cdots,n$) such that the interfaces $\Gamma_t$ are given by 
		\be\label{interface'}
		\Gamma_t=\left\{x\in\R^N;\ \min_{1\le i\le n}\{x\cdot e_i +\xi_t^i\}=0\right\}.
		\ee
		Then there exist $m\in \{1,\cdots,n\}$ vectors $e_i$ and constants $\tau_i$ ($i=1,\cdots,m$) such that $e_i\cdot e_0>0$ for a certain $e_0\in \mathbb{S}^{N-1}$ and
		\[u(t,x)\equiv U(t,x;e_1,\cdots,e_m,\tau_1,\cdots,\tau_m), \hbox{ for $t\in\R$ and $x\in\R^N$}.\]
	\end{theorem}
	
The paper is organized as follows. In Section~3, we construct entire solutions by mixing finite planar fronts and prove their asymptotic behaviour, monotonicity and uniqueness. In Section~4, we prove the stability of the entire solution constructed in Section~3. Section~5 is devoted to the characterization problem of transition fronts whose level sets are asymptotic to finitely many hyperplanes.

{\bf Notations.} The following notations will be used throughout the paper.
\begin{itemize}
    \item 	By \eqref{F}, there exists $\sigma\in (0,1/8)$ such that
		\be\label{sigma}
		f^\prime(u)\le \frac{f^\prime(0)}{2} \hbox{ for $u\in [0,4\sigma]$ and } f^\prime(u)\le \frac{f^\prime(1)}{2} \hbox{ for $u\in [1-4\sigma,1]$}.
		\ee
  \item  $L:=\max_{0\le u\le 1} |f^\prime(u)|$.
\item 		Since $g(-\infty)=1$ and $g(+\infty)=0$, there exists an $R>0$ such that 
\be\label{eq-R}
0<g(\xi)\le \sigma \hbox{ for $\xi\ge R$} \quad \hbox{ and}  \quad 1-\sigma\le g(\xi)<1 \hbox{ for $\xi\le -R$}.
\ee
\item 		Since $g^\prime<0$, there exists $k>0$ such that $-g^\prime\ge k$ in $[-R,R]$.
\item From \cite{FM}, one knows that $|g^\prime|$ and $|g^{\prime\prime}|$ decay exponentially as $|\xi|\rightarrow +\infty$. Thus, there exists $M>0$ such that 
		\be\label{M}
		|g^\prime(\xi)|+|g^\prime(\xi)\xi|+|g^{\prime\prime}(\xi)\xi|+|g^{\prime\prime}(\xi)\xi^2|\le M, \quad \forall  \xi\in \R.
		\ee
  \item $\mu:=\min\left\{1,-\frac{f^\prime(0)}{4},-\frac{f^\prime(1)}{4}\right\}$.

  \end{itemize}

\section{Existence and uniqueness: Proof of Theorem \ref{Th1}}

This section is devoted to the proof of Theorem~\ref{Th1}, that is, we construct an entire solution from the data $\{(e_i,\tau_i)\}_{i=1,\cdots,n}$. Since the equation \eqref{RD} is invariant under rotation of  coordinates, we will assume without loss of generality that $e_0=(0,0,\cdots,1)$, that is, the $x_N$-direction. By denoting $y:=x_N$, \eqref{RD} is rewritten as
\be\label{RDxy}
\partial_tu-\Delta_x u-u_{yy}=f(u),\ t\in\R,\ x\in\R^{N-1},\ y\in\R.
\ee

\subsection{A surface with asymptotic planes}\label{subsection:phi}

In this subsection, we introduce a hypersurface which is asymptotic  to several hyperplanes at infinity. We will use the signed distance function to this hypersurface to construct supersolutions in the next subsection.

Because	$e_i\cdot e_0>0$ for each  $i=1$, $\dots$, $n$, there exist a unit vector $\nu_i\in \mathbb{S}^{N-2}$ and an angle $\theta_i$ of $(0,\pi/2]$  such that $e_i=(\nu_i\cos\theta_i,\sin\theta_i)$.  Recall that $\tau_i$ ($i=1,2,\cdots,n$) are constants. Take a positive constant $\alpha$, which will be used as a  scaling parameter later on. For any  $(t,x,y)\in\R\times\R^N$ and every $i\in\{1,\cdots,n\}$, denote
	\[q_i(t,x,y)=x\cdot \nu_i \cos\theta_i +y \sin\theta_i -c_f t +\alpha\tau_i.\]
Associated to each $q_i$ is the hyperplane in $\R\times\R^N$,
	\[Q_i:=\left\{(t,x,y)\in\R\times\R^N; q_i(t,x,y)=0\right\}.\]
It is a graph in the $y$-direction, $\{y=\psi_i(t,x)\}$, where
	\[\psi_i(t,x)=-x\cdot\nu_i\cot\theta_i+\frac{c_f}{\sin\theta_i}t-\frac{\alpha}{\sin\theta_i}\tau_i.\]
Notice that $q_i$ is the signed distance function to $Q_i$.
Let $\mathcal{Q}$ be the polytope enclosed by $Q_1$, $\dots$, $Q_n$, that is,
	\[\mathcal{Q}:=\left\{(t,x,y)\in\R\times\R^N; \min_{1\le i\le n}q_i(t,x,y)\ge 0\right\}.\]
Let $\partial\mathcal{Q}$ be the boundary of $\mathcal{Q}$. It has the form $\{y=\psi(t,x)\}$, where
	\[\psi(t,x):=\max_{1\le i \le n}\psi_i(t,x).\]
As a consequence, $\partial\mathcal{Q}=\cup\widetilde{Q}_i$, where $\widetilde{Q}_i=\partial\mathcal{Q}\cap Q_i$ are its facets. If we denote
	\[ \Psi(t,x,y):=\min_{1\leq i \leq n}q_i(t,x,y)=\Phi(x,y)-c_ft,\]
then $\partial\mathcal{Q}=\{\Phi=c_ft\}$. That is, for any $t$, $\partial\mathcal{Q}_t$ is the $c_ft$ level set of $\Phi$. Let $G_{ij}=\widetilde{Q}_{i}\cap\widetilde{Q}_j$ be the ridge as the intersection of $\widetilde{Q}_i$ and $\widetilde{Q}_j$ for some $i\neq j$. Let $G$ be the set of all ridges of $\mathcal{Q}$ and $\widehat{G}$ be the projection of $G$ on the $(t,x)$-plane, that is, 
	\[\widehat{G}:=\left\{(t,x)\in\R\times\R^{N-1};\ \hbox{ there exists one $y\in\R$ such that} (t,x,y)\in G\right\}.\]
	Let $\widehat{Q}_i$ be the projection of $\widetilde{Q}_i$ on the $(t,x)$-plane, that is,
	\begin{eqnarray*}\widehat{Q}_i&:=&\left\{(t,x)\in\R\times\R^{N-1};\ \hbox{ there exist $y\in\R$ such that }  (t,x,y)\in\widetilde{Q}_i\right\}\\
		&=&\left\{(t,x)\in\R\times\R^{N-1};\ \psi_i(t,x)=\max_{1\leq j\leq n}\psi_j(t,x)\right\}.
	\end{eqnarray*}
	By the graph property of $\partial\mathcal{Q}$, $\cup_{i=1}^n \partial\widehat{Q}_i=\widehat{G}$ and $\cup_{i=1}^n \widehat{Q}_i=\R\times\R^{N-1}$.

Let $y=\varphi(t,x)$ be the function determined by the relation
	\be\label{surface}
	\sum_{i=1}^{n} e^{-q_i(t,x,y)}=1.
	\ee
	The existence of such a function is guaranteed by the implicit function theorem, which also implies that $\varphi \in C^{\infty}(\R\times\R^{N-1})$. Let  $\Sigma:=\{y=\varphi(t,x)\}$    be the graph of $\varphi$. Hereafter, we use $\widehat{q}_i(t,x)$ to denote $q_i(t,x,\varphi(t,x))$ for short. We also set
	\[h:=\sum_{i, j\in\{1,\cdots,n\},i\neq j } e^{-(\widehat{q}_i+\widehat{q}_j)}.\]
	The following lemma  shows that $h$ is a measurement of flatness for $\Sigma$.
	\begin{lemma}
	The graph $\Sigma$ satisfies the following properties:
		\begin{description}
			\item [(i)] $\Sigma\subset \mathcal{Q}$;
			\item [(ii)]  $\Sigma$ stays at finite distance from $\partial \mathcal{Q}$, or equivalently $\sup_{\R\times\R^N}|\varphi-\psi|<+\infty$;
			\item [(iii)] $\Sigma$ approaches $\partial \mathcal{Q}$ exponentially away from $G$, or equivalently, there exists a constant $C$ such that
			\be\label{expoential close}
			|\varphi(t,x)-\psi(t,x)|\leq C\exp\Big\{-\frac{1}{C}d(x,G)\Big\} \text{ or } C h(t,x),\quad \forall (t,x)\in\R\times\R^N.
			\ee
		\end{description}
	\end{lemma}
	
	\begin{proof}
		(i) By  \eqref{surface}, for every point $(t,x,y)\in\Sigma$, 
		\[\min_{1\le i\le n} q_i(t,x,y)\geq 0.\]
		Hence $(t,x,y)\in \mathcal{Q}$.

		(ii) If this is not true, because $q_i$ is the signed distance to $Q_i$, there would exist a sequence $\{(t_k,x_k)\}_{k\in\mathbb{N}}$ of $\R\times\R^{N-1}$ such that
		\[\min_{1\le i\le n} q_i(t_k,x_k,\varphi(t_k,x_k))\rightarrow +\infty \hbox{ or } -\infty \hbox{ as } k\rightarrow +\infty.\]
		This means $\sum_{i=1}^n e^{-q_i(t_k,x_k,\varphi(t_k,x_k))}\rightarrow 0$ or $+\infty$ respectively as $k\rightarrow +\infty$, which is a contradiction.
		
		(iii) 
		In this step we fix an $i\in \{1,2,\cdots,n\}$, and assume  $(t,x)\in\widehat{Q}_i$. By (i) and the definition of $\mathcal{Q}$, we have
		\be\label{2.0}
		\varphi(t,x)\ge \psi_i(t,x).
		\ee
		
		We claim that there exists a positive constant $C^\prime$ such that
		\be\label{2.1}
		\min_{j\neq i}q_j(t,x,\varphi(t,x))\ge C^\prime d((t,x),\widehat{G}).
		\ee
		Indeed, by \eqref{2.0}, for any $j\neq i$ we have
		\[q_j(t,x,\varphi(t,x))\ge q_j(t,x,\psi_i(t,x)).\]
		On the other hand, by the definition of $\psi_i$, we have
		\[q_i(t,x,\psi_i(t,x))=0.\]
		Because $e_j\neq e_i$, we find a positive constant $C^\prime$ such that
		\[q_j(t,x,\psi_i(t,x))\geq C^\prime\hbox{dist}\left((t,x), \{\psi_i=\psi_j\}\right),\hbox{ for all $j\neq i$},\]
		and \eqref{2.1} follows.
		
		Combining \eqref{2.1} with the definition of $\Sigma$, we see if  $d((t,x),\widehat{G})\rightarrow +\infty$, then 
		\[q_i(t,x,\varphi(t,x))\rightarrow 0.\]
		In other words, as  $d((t,x),\widehat{G})\rightarrow +\infty$,  
		\be\label{app-facets}
		\left|\varphi(t,x)-\psi_i(t,x)\right|\rightarrow 0 .
		\ee
		In $\widehat{Q}_i$, by definition one has
		\begin{align*}
			\varphi-\psi=\varphi-\psi_i=-\frac{1}{\sin\theta_i}\ln \left(1-\sum_{j\neq i}e^{-\widehat{q}_j}\right).
		\end{align*}
		Hence in view of \eqref{2.1}, there exists a constant $C>0$ such that in $\widehat{Q}_i$, 
  \be\label{varphi-plane}
		\frac{1}{C}\sum_{j\neq i}e^{-\widehat{q}_j}\le \varphi-\psi\le C\sum_{j\neq i}e^{-\widehat{q}_j}.
		\ee
		Using \eqref{2.1} again, we also find another positive constant $C$ such that in $\widehat{Q}_i$, 
		\be\label{2.2}
		\frac{1}{C}h\le \sum_{j\neq i}e^{-\widehat{q}_j}\le Ch.
		\ee
		Combining \eqref{2.1} and \eqref{2.2}, we get \eqref{expoential close}.
	\end{proof}
\vskip 0.3cm
	
	We now collect some estimates for derivatives of $\varphi$.  
	\begin{lemma}\label{pro-phi}
		There exists a constant  $C>0$ independent of $\alpha$ and $\tau_i$ such that for each $i\in\{1,\cdots,n\}$, 
		\[
		\left|\partial_t\varphi-\frac{c_f}{\sin\theta_i}\right|+
		\left|\nabla\varphi+\nu_i \cot\theta_i\right|\le Ch  \quad 
		\text{in} ~~ \widehat{Q}_i,
		\]
		\be\label{xi-0}
		\frac{1}{C} h\le \frac{\partial_t\varphi}{\sqrt{1+|\nabla\varphi|^2}}-c_f\le C h \quad  \text{in} ~~ \R\times\R^{N-1},
		\ee
		and
		\[
		\left|\nabla\partial_t\varphi\right|+\left|\nabla^2\varphi\right|+\left|\nabla^3 \varphi\right|\le Ch \quad  \text{in} ~~ \R\times\R^{N-1}.
		\]
	\end{lemma}
	
	\begin{proof}
		Substituting $y=\varphi(t,x)$ into \eqref{surface} and differentiating it in $t$ and $x$, one gets
		\be\label{phitnablaphi} 
		\partial_t\varphi=\frac{c_f}{\sum_{i=1}^{n}e^{-\widehat{q}_i}\sin\theta_i} ,\qquad \nabla\varphi=-\frac{\sum_{i=1}^{n}e^{-\widehat{q}_i}\nu_i\cos\theta_i}{\sum_{i=1}^{n}e^{-\widehat{q}_i}\sin\theta_i},
		\ee
		\[\nabla\partial_t\varphi=\frac{\sum_{i=1}^{n}e^{-\widehat{q}_i}(\nu_i\cos\theta_i+\nabla\varphi\sin\theta_i)(\partial_t\varphi\sin\theta_i -c_f)}{\sum_{i=1}^{n}e^{-\widehat{q}_i}\sin\theta_i},\]
		and
		\begin{align*}
			&\nabla^2\varphi=\left(\partial_{x_k}\partial_{x_l}\varphi\right)\\
			=&\left(\frac{\sum_{i=1}^{n}e^{-\widehat{q}_i}\sum_{j=1}^{n}\sum_{s=1}^{n}e^{-(\widehat{q}_j+\widehat{q}_s)}( \nu_i^k\cos\theta_i \sin\theta_s - \nu_s^k\sin\theta_i\cos\theta_s )( \nu_i^l\cos\theta_i \sin\theta_j - \nu_j^l \sin\theta_i\cos\theta_j )}{(\sum_{i=1}^{n}e^{-\widehat{q}_i}\sin\theta_i)^3}\right),
		\end{align*}
		where $\nu_i^k$ denotes the $k$-th component of $\nu_i$.  
		
		Since $\min_{1\le i\le n} q_i(t,x,\varphi(t,x))\ge 0$ and $\theta_i\in (0,\pi/2]$, one sees that $\partial_t\varphi$, $|\nabla \varphi|$, $|\nabla\partial_t\varphi|$ and $|\nabla^2 \varphi|$ are uniformly bounded on $\R\times\R^{N-1}$. Moreover, because $(\nu_i,\theta_i)\neq (\nu_j,\theta_j)$ for $i\neq j$, there exists a positive constant $C$ such that in $\widehat{Q}_i$ ,
		\be\label{phit}
		\left|\partial_t\varphi-\frac{c_f}{\sin\theta_i}\right|=\left |\frac{\sum_{j=1}^{n}e^{-\widehat{q}_j}(\sin\theta_i-\sin\theta_j)}{\sum_{j=1}^{n}e^{-\widehat{q}_j}\sin\theta_j\sin\theta_i} c_f \right|\le Ch,
		\ee
		\be\label{phix}
		|\nabla\varphi+\nu_i \cot\theta_i|=\left|\frac{\sum_{j=1}^{n}e^{-\widehat{q}_j}(\nu_i\sin\theta_j\cos\theta_i-\nu_j\cos\theta_j\sin\theta_i)}{\sum_{j=1}^{n}e^{-\widehat{q}_j}\sin\theta_j\sin\theta_i}\right|\le C h,
		\ee
		and in the whole $\R\times\R^{N-1}$,
		\be\label{phixx}
		|\nabla\partial_t\varphi|+|\nabla^2\varphi|\le C h.
		\ee
		The inequality \eqref{phixx} for $|\nabla^3\varphi|$ follows in a similar way.
		
		By \eqref{phitnablaphi} and the fact that $\sum_{i=1}^n e^{-\widehat{q}_i}=1$, we have
		\begin{align*}
			\frac{\partial_t\varphi}{\sqrt{1+|\nabla\varphi|^2}}-c_f
			=&\left(\frac{1}{\sqrt{(\sum_{i=1}^n e^{-\widehat{q}_i}\sin\theta_i)^2+(\sum_{i=1}^n e^{-\widehat{q}_i}\nu_i\cos\theta_i)^2}}-1\right)c_f\\
			=& c_f\frac{1-\sqrt{1-\sum_{i=1}^n \sum_{j=1}^n e^{-(\widehat{q}_i+\widehat{q}_j)}(1-e_i\cdot e_j)}}{\sqrt{(\sum_{i=1}^n e^{-\widehat{q}_i}\sin\theta_i)^2+(\sum_{i=1}^n e^{-\widehat{q}_i}\nu_i\cos\theta_i)^2}}\\
			=& c_f\frac{1-\sqrt{1-\sum_{i\neq j}^n e^{-(\widehat{q}_i+\widehat{q}_j)}(1-e_i\cdot e_j)}}{\sqrt{(\sum_{i=1}^n e^{-\widehat{q}_i}\sin\theta_i)^2+(\sum_{i=1}^n e^{-\widehat{q}_i}\nu_i\cos\theta_i)^2}}.
		\end{align*}
		Since $\theta_i\in (0,\pi/2]$ and $\sum_{i=1}^n e^{-\widehat{q}_i}=1$, there exists a universal constant $C>0$ such that 
		\[\left(\sum_{i=1}^n e^{-\widehat{q}_i}\sin\theta_i\right)^2+\left(\sum_{i=1}^n e^{-\widehat{q}_i}\nu_i\cos\theta_i\right)^2\le C.\]
		Moreover, because $e_i\neq e_j$ for $i\neq j$, $1-e_i\cdot e_j>0$. Thus
		\[1-\sqrt{1-\sum_{i\neq j}^n e^{-(\widehat{q}_i+\widehat{q}_j)}(1-e_i\cdot e_j)}\ge C\sum_{i\neq j}^n e^{-(\widehat{q}_i+\widehat{q}_j)}, \]
for some $C>0$ and \eqref{xi-0}  follows.
	\end{proof}

\begin{remark}
 From the above proof, we see that the Hessian matrix $\nabla^2\varphi$ is positive definite. Hence the surface $\{y=\varphi(t,x)\}$ is convex in $\R^N$ for every $t$. 
\end{remark}

\subsection{A pair of sub- and super-solutions}\label{subsection:sup-sub}
	
First let
	\[\underline{u}(t,x,y):= \max_{1\le i\le n} g(x\cdot \nu_i\cos\theta_i +y\sin\theta_i -c_f t +\tau_i).\]
Because for each $i$, $g(x\cdot \nu_i\cos\theta_i +y\sin\theta_i -c_f t +\tau_i)$ is a solution of \eqref{RDxy}, $\underline{u}$ is a subsolution.

The construction of a supersolution requires more work. 
For  any $\alpha>0$ and any $\varepsilon>0$, define 
	\[\overline{u}(t,x,y):=\min\left\{g\left(\overline{\xi}(t,x,y)\right) +\varepsilon h(\alpha t,\alpha x),1 \right\},\]
	where
	\[\overline{\xi}(t,x,y)=\frac{y-\frac{1}{\alpha}\varphi (\alpha t, \alpha x)}{\sqrt{1+|\nabla\varphi (\alpha t, \alpha x)|^2}}.\]
	In the sequel, we will denote $\varphi (\alpha t, \alpha x)$ by $\varphi$ for short. But one should keep in mind that $\varphi$ and its derivatives take values at $(\alpha t, \alpha x)$. We also emphasize that $\underline{u}$ does not depend on the parameter $\alpha$ and $\varepsilon$, while $\overline{u}$ depends on them, although we will not write this dependence explicitly.

	A direct computation shows that
	\[-\partial_t\overline{\xi} -c_f =\frac{\partial_t\varphi}{\sqrt{1+|\nabla\varphi|^2}}-c_f +\alpha \overline{\xi}\frac{ \nabla\varphi\cdot \nabla\partial_t\varphi}{1+|\nabla\varphi|^2}.\]
	By Lemma~\ref{pro-phi}, one has the following lemma.
	
	\begin{lemma}\label{pro-xi}
		There exists a positive constant $C$  such that for $(t,x,y)\in\R\times\R^{N}$,
		\be\label{xi-1}
		\frac{1}{C} h(\alpha t,\alpha x)\le \frac{\partial_t\varphi}{\sqrt{1+|\nabla\varphi|^2}}-c_f\le C h(\alpha t,\alpha x),
		\ee
		\be\label{xi-2}
		\left|\frac{\alpha \nabla\varphi\cdot \nabla\partial_t\varphi}{1+|\nabla\varphi|^2}\overline{\xi}\right|\le  C \alpha h(\alpha t,\alpha x)|\overline{\xi}|, 
		\ee
		\be\label{xi-3}
		\left|1-|\nabla_x\overline{\xi}|^2 -|\partial_y \overline{\xi}|^2\right|\le  C \alpha h(\alpha t,\alpha x)\left (|\overline{\xi}|+\overline{\xi}^2\right),
		\ee
		\be\label{xi-4}
		|\Delta_x \overline{\xi}|\le C \alpha h(\alpha t,\alpha x)(1+|\overline{\xi}|+ \alpha |\overline{\xi}|).
		\ee
	\end{lemma}
	
	\begin{proof}
		The first two inequalities are just scalings of those in  Lemma~\ref{pro-phi}.
		
		To obtain \eqref{xi-3} and \eqref{xi-4},  we use
		\begin{align*}
			&1-\left|\nabla_x\overline{\xi}\right|^2 -\left|\partial_y \overline{\xi}\right|^2\\
			=& 1-\left|\frac{\nabla \varphi}{\sqrt{1+|\nabla\varphi|^2}}+\frac{\overline{\xi}}{1+|\nabla\varphi|^2}\alpha\nabla\varphi\cdot \nabla^2\varphi \right|^2 -\frac{1}{1+|\nabla\varphi|^2}\\
			=& -2\alpha \overline{\xi}\frac{\nabla\varphi\cdot (\nabla^2\varphi \cdot \nabla\varphi)}{(1+|\nabla\varphi|^2)^{3/2}} -\alpha^2\overline{\xi}^2\frac{\left|\nabla\varphi\cdot \nabla^2\varphi\right|^2}{(1+|\nabla\varphi|^2)^2},
		\end{align*}
		and
		\begin{eqnarray*}
			\Delta_x \overline{\xi}&=&-\frac{\alpha\Delta \varphi}{\sqrt{1+|\nabla\varphi|^2}} +2\alpha\frac{\nabla\varphi\cdot (\nabla^2\varphi \cdot\nabla \varphi)}{(1+|\nabla\varphi|^2)^{3/2}} +3\alpha \overline{\xi}\frac{|\nabla^2\varphi\cdot \nabla\varphi|^2}{(1+|\nabla\varphi|^2)^{3/2}}\\ &-&\alpha^2\overline{\xi}\frac{ \nabla\cdot (\nabla^2\varphi\cdot \nabla\varphi)}{1+|\nabla\varphi|^2}.
		\end{eqnarray*}
		By substituting the estimates on $\nabla^2\varphi$ etc. from Lemma \ref{pro-phi} into these formulas, we obtain \eqref{xi-3} and \eqref{xi-4}.
	\end{proof}
	\vskip 0.3cm
	
	With these preliminary estimates, now we show that $\overline{u}$ is indeed a supersolutoin of \eqref{RDxy}. In the following we say a function $\alpha$ is a modulus function if it is defined on an interval $(0,\varepsilon_\ast)$, it is increasing and satisfies
 \[\lim_{\varepsilon\to0} \alpha(\varepsilon)=0.\]
	\begin{lemma}\label{lem supersolution}
		There exist $\varepsilon_0$ and a modulus function $\alpha(\varepsilon)$ defined on $(0,\varepsilon_0)$ such that for any $0<\varepsilon<\varepsilon_0$ and $0<\alpha\le\alpha(\varepsilon)$, the function $\overline{u}$ is a supersolution of \eqref{RDxy}  in  $\R \times\R^{N}$. 
	\end{lemma}
	\begin{proof}	
		By the definition of $h$, one has
		\[\partial_th(t,x)=c_f h(t,x),\]
		Moreover,  there exists a constant $C>0$ such that
		\begin{align*}
			|\Delta h(t,x)|=&\left|\sum_{i, j\in\{1,\cdots,n\},i\neq j } e^{-(\widehat{q}_i+\widehat{q}_j)}\left|\nu_i\cos\theta_i +\nu_j\cos\theta_j +\nabla\varphi \sin\theta_i\right|^2 +\sum_{i, j\in\{1,\cdots,n\},i\neq j } e^{-(\widehat{q}_i+\widehat{q}_j)}\Delta \varphi (\sin\theta_i +\sin\theta_j)\right|\\
			\le& C h(t,x).
		\end{align*}
After enlarging the constant $C$ if necessary, we can make sure that Lemma~\ref{pro-xi} still holds. 
	 Let 
		\[\varepsilon_0=\min \left\{\frac{\sigma}{n^2},\frac{k }{2CL},1\right\},\]
and for $0<\varepsilon<\varepsilon_0$,
		\[\alpha(\varepsilon)=\min\left\{1,-\frac{f^\prime(0)\varepsilon}{2(6C M + c_f + C)},-\frac{f^\prime(1)\varepsilon}{2(6C M + c_f +   C)}\right\}.\]
		
		To show that $\overline{u}$ is a supersolution,	one needs  only to consider the domain where $\overline{u}<1$. By the definition of $\overline{u}$,  Lemma~\ref{pro-xi} and \eqref{M}, in $\{\overline{u}<1\}$ it holds that
		\begin{eqnarray*}
			N(t,x,y)&:=&\partial_t\overline{u} -\Delta \overline{u} -f(\overline{u})\\
			&=&g^\prime(\overline{\xi})\left(\partial_t\overline{\xi} +c_f\right) +g^{\prime\prime}(\overline{\xi})\left(1-|\nabla_x \overline{\xi}|^2-|\partial_y\overline{\xi}|^2\right)-g^\prime(\overline{\xi})\Delta_x\overline{\xi} \\
			&&+f(g(\overline{\xi})) -f(\overline{u})+\varepsilon \alpha \partial_th(\alpha t,\alpha x)  -\varepsilon \alpha^2 \Delta h(\alpha t,\alpha x) \\
			&\ge& -g^\prime(\overline{\xi})\frac{1}{C} h(\alpha t,\alpha x) -M \alpha C h(\alpha t,\alpha x) -2M \alpha C h(\alpha t,\alpha x) -3M\alpha C h(\alpha t,\alpha x) \\
			&&+f(g(\overline{\xi})) -f(\overline{u})+\varepsilon \alpha c_f h(\alpha t,\alpha x) -\varepsilon \alpha^2 C h(\alpha t,\alpha x) \\
			&\ge & -g^\prime(\overline{\xi}) \frac{1}{C} h(\alpha t,\alpha x) -\alpha  h(\alpha t,\alpha x) (6C M +   c_f +   C)+f(g(\overline{\xi})) -f(\overline{u}),
		\end{eqnarray*}
		provided that $0<\varepsilon<\varepsilon_0$ and $0<\alpha\le \alpha(\varepsilon)$. 
		
		In $\{\overline{\xi}\ge R\}$,  if $0<\varepsilon<\varepsilon_0$, then
		\[\overline{u}=g(\overline{\xi})+\varepsilon h(\alpha t,\alpha x)\le 2\sigma.\]   
  Here, we have used the fact  that $h(t,x)\le n^2$. Therefore
		\[f(g(\overline{\xi})) -f(\overline{u})\ge -\frac{f^\prime(0)}{2} \varepsilon h(\alpha t,\alpha x).\]
		Then by the facts that $g^\prime<0$ and $0<\alpha\le \alpha(\varepsilon)$, we deduce that $N\ge 0$ in $\{\overline{\xi}\ge R\}$.  In the same way, we can show that $N\ge 0$ in $\{\overline{\xi}\le -R\}$. 
		
		In $\{-R\le \overline{\xi}\le R\}$,  $-g^\prime\ge k$. Moreover,
		\[f(g(\overline{\xi})) -f(\overline{u})\ge -L \varepsilon h(\alpha t,\alpha x).\]
		Therefore, using the facts that $0<\varepsilon<\varepsilon_0$ and $0<\alpha\le \alpha(\varepsilon)$, we get
		\[N(t,x,y)\ge k  \frac{1}{C} h(\alpha t,\alpha x)-\alpha  h(\alpha t,\alpha x)(6C M +   c_f +   C)-L \varepsilon h(\alpha t,\alpha x)\ge 0.\]
		
		Putting  estimates in these three regions together, we deduce that $N\ge 0$ everywhere in $\{\overline{u}<1\}$. 
	\end{proof}
\vskip 0.3cm

Next, we need to compare $\overline{u}$ and $\underline{u}$. Recall that $\mathcal{P}$ and $\mathcal{R}$ are the polytope and its ridges without the scaling parameter $\alpha$ defined in Section~2.

	\begin{lemma}\label{lem comparison of sup and sub}
		\begin{enumerate}
			\item[(i)] In $\R\times\R^N$,	$\overline{u}> \underline{u}$.
			\item[(ii)] For any $\varepsilon\in(0,\varepsilon_0)$, there exists an $\rho(\varepsilon)$ such that in $\{d((t,x,y),\partial\mathcal{P})\ge \rho(\varepsilon)\}$ , \[|\overline{u}(t,x,y)-\underline{u}(t,x,y)|\le (n^2+1)\varepsilon.\]
			\item [(iii)] For any $M>0$  fixed, in $\{d((t,x,y),\partial\mathcal{P})\le M\}$ , 
			\be\label{coincide at infinity}
			|\overline{u}(t,x,y)-\underline{u}(t,x,y)|\rightarrow 0 \hbox{ uniformly as $d((t,x,y),\mathcal{R})\rightarrow +\infty$}.
			\ee
		\end{enumerate}
	\end{lemma}
	
	\begin{proof}
		(i) Fix an $i\in\{1,\cdots,n\}$. Let 
  \[\xi_i(t,x,y)=x\cdot \nu_i\cos\theta_i +y\sin\theta_i -c_f t +\tau_i.\]
 
  In $\{\overline{\xi}\le \xi_i\}$, because $g$ is a decreasing function, one has
		\[
		\overline{u}(t,x,y)=g(\overline{\xi})+\varepsilon h(\alpha t,\alpha x)\ge g(\xi_i)=g(x\cdot \nu_i\cos\theta_i +y\sin\theta_i -c_f t +\tau_i).\]
  
 In $\hat{P}_i$ (the projection of $\widetilde{P}_i$ on the $(t,x)$-plane), $\underline{u}=g(\xi_i)$. So 
if here we have $\overline{\xi}(t,x,y)> \xi_i(t,x,y)$, then
\begin{align*}
			0<&\overline{\xi}(t,x,y)-\xi_i(t,x,y)\\
			=&\frac{y-\frac{1}{\alpha}\varphi(\alpha t,\alpha x)}{\sqrt{1+|\nabla\varphi|^2}}-(x\cdot \nu_i\cos\theta_i +y\sin\theta_i -c_f t +\tau_i)\\
			=& \left(\frac{1}{\sin\theta_i\sqrt{1+|\nabla\varphi|^2}}-1\right)\xi_i(t,x,y) -\frac{\alpha x\cdot \nu_i\cot\theta_i -\alpha c_f  t/\sin\theta_i +\alpha \tau_i/\sin\theta_i +\varphi(\alpha t,\alpha x)}{\alpha\sqrt{1+|\nabla\varphi|^2}}.
		\end{align*}
Combining this inequality with \eqref{expoential close} and Lemma~\ref{pro-phi}, we see that there exists a positive, universal constant $A_1$ such that 
		\[|\xi_i(t,x,y)|\ge \frac{h(\alpha t,\alpha x)}{\alpha \sqrt{1+|\nabla\varphi|^2}}\left|\frac{1}{\sin\theta_i \sqrt{1+|\nabla\varphi|^2}}-1\right|^{-1}\ge \frac{A_1}{\alpha}.\]
		Since the surface $\{y=\varphi(t,x)\}$ lies inside the polytope $\mathcal{Q}$, that is, $\varphi(\alpha t,\alpha x)\ge \psi_i(\alpha t,\alpha x)\ge 0$ for every $i\in\{1,\cdots,n\}$, we have
		\[\xi_i(t,x)<\overline{\xi}(t,x,y)=\frac{y-\frac{1}{\alpha}\varphi(\alpha t,\alpha x)}{\sqrt{1+|\nabla\varphi|^2}}\le \frac{x\cdot \nu_i\cos\theta_i +y\sin\theta_i -c_f t +\tau_i}{\sin\theta_i\sqrt{1+|\nabla\varphi|^2}}=\frac{\xi_i(t,x,y)}{\sin\theta_i\sqrt{1+|\nabla\varphi|^2}}.\]
Hence
		\begin{align*}
			\overline{u}(t,x,y)-g(\xi_i)=&g(\overline{\xi})-g(\xi_i)+\varepsilon h(\alpha t,\alpha x)\\
   =& g^\prime(\xi^\prime) \left(\overline{\xi}-\xi_i\right)+\varepsilon h(\alpha t,\alpha x)\\
			\ge& g^\prime(\xi^\prime)\xi_i\Big(\frac{1}{\sin\theta_i\sqrt{1+|\nabla\varphi|^2}}-1\Big)+\varepsilon h(\alpha t,\alpha x)\\
			\ge& -\left|g^\prime(\xi^\prime)\xi_i^2\right| \frac{\alpha}{A_1} h(\alpha t,\alpha x)+\varepsilon h(\alpha t,\alpha x)
		\end{align*}
where $\xi^\prime$ is between $\xi_i$ and $\xi_i/(\sin\theta_i \sqrt{1+|\nabla\varphi|^2})$. Since $|g^\prime(\xi^\prime)\xi_i^2|$ is bounded, one can make $\alpha(\varepsilon)$ smaller so that $\overline{u}(t,x,y)-g(\xi_i)\ge 0$.  Because  $\cup_{i} \hat{P}_i=\R\times\R^{N-1}$, we then deduce that for any $(t,x,y)\in\R\times\R^N$,
		\[\overline{u}(t,x,y)\ge \max_{i\in\{1,\cdots,n\}}\left\{g(x\cdot \nu_i\cos\theta_i +y\sin\theta_i -c_f t +\tau_i)\right\}=\underline{u}(t,x,y).\]
Then by the strong maximum principle, we deduce that $\overline{u}>\underline{u}$ in $\R\times \R^N$.
		
		(ii)  Fix an $\alpha\in (0,\alpha(\varepsilon)]$.  Assume by the contrary that there exists an  $\varepsilon\in (0,\varepsilon_0)$ and a sequence of points $(t_k,x_k,y_k)$  such that $d((t_k,x_k,y_k),\partial\mathcal{P})\to+\infty$, but
\[|\overline{u}(t_k,x_k,y_k)-\underline{u}(t_k,x_k,y_k)|> (n^2+1)\varepsilon.\]
There are two cases.

{\bf Case 1.} The sequence  $(t_k,x_k,y_k)$ satisfies
		\be\label{c+infty}
		\min_{1\le i\le n}\left\{x_k\cdot \nu_i \cos\theta_i +y_k\sin\theta_i-c_f t_k + \tau_i\right\}\rightarrow +\infty,
		\ee

		This implies $\underline{u}(t_k,x_k,y_k)\rightarrow 0$. Since $\alpha>0$, the convergence \eqref{c+infty} still holds for $(t_k,x_k,y_k)$ replaced by $(\alpha t_k,\alpha x_k,\alpha y_k)$ and $\tau_i$ replaced by $\alpha \tau_i$.
		Thus $(\alpha t_k,\alpha x_k,\alpha y_k)\in \mathcal{Q}$ and  
		\[d((\alpha t_k,\alpha x_k,\alpha y_k),\partial\mathcal{Q})\rightarrow +\infty \hbox{ as $k\rightarrow +\infty$}.\]
		 Since the surface $\{y=\varphi(t,x)\}$ is bounded away from $\partial\mathcal{Q}$, one has that $\alpha y_k -\varphi(\alpha t_k ,\alpha x_k)\rightarrow +\infty$ and hence, $\overline{\xi}(t_k,x_k,y_k)\rightarrow +\infty$. This implies that $g(\overline{\xi}(t_k,x_k,y_k))\rightarrow 0$ and
		\[|\overline{u}(t_k,x_k,y_k)-\underline{u}(t_k,x_k,y_k)|\le \varepsilon \left(h(\alpha t_k,\alpha x_k)+\frac{1}{2}\right)\le \left(n^2+\frac{1}{2}\right)\varepsilon,  \hbox{ for $k$ large enough}.\]
  This is a contradiction.
  
	{\bf Case 2.} The sequence  $(t_k,x_k,y_k)$ satisfies
		\be \label{c-infty}
		\min_{1\le i\le n}\left\{x_k\cdot \nu_i \cos\theta_i +y_k\sin\theta_i-c_f t_k + \tau_i\right\}\rightarrow -\infty.
		\ee
  In this case, one can prove in the same way that $g(\overline{\xi}(t_k,x_k,y_k))\rightarrow 1$ and $\underline{u}(t_k,x_k,y_k)\rightarrow 1$ as $k\rightarrow +\infty$. This leads to the same contradiction.
		
		(iii) If $d((t,x,y),\partial\mathcal{P})\leq M$, there exists $i\in\{1,\cdots,n\}$ such that   as $d((t,x,y),\mathcal{R})\rightarrow +\infty$,
		\be\label{ibju}
		|x\cdot \nu_i \cos\theta_i +y\sin\theta_i-c_f t + \tau_i| \hbox{ is bounded and } x\cdot \nu_j \cos\theta_j +y\sin\theta_j-c_f t + \tau_j\rightarrow +\infty \hbox{ for $j\neq i$},
		\ee
		Hence $\underline{u}(t,x,y)=g(x\cdot \nu_i \cos\theta_i +y\sin\theta_i-c_f t + \tau_i)$. Since $\alpha>0$, \eqref{ibju} also holds for $(t,x,y)$ replaced by $(\alpha t,\alpha x,\alpha y)$ and $\tau_i$ replaced by $\alpha \tau_i$. In particular,
		\be\label{pibpju}
		q_i(\alpha t,\alpha x,\alpha y)\hbox{ is bounded and } q_j(\alpha t,\alpha x,\alpha y) \rightarrow +\infty \hbox{ for $j\neq i$}.
		\ee
	This implies that $(\alpha t,\alpha x,\alpha y)$ is bounded away from $\widetilde{Q}_i$ and $d((\alpha t,\alpha x,\alpha y),G)\rightarrow +\infty$. It also implies that $(\alpha t,\alpha x)\in \widehat{Q}_i$ and $d((\alpha t,\alpha x),\widehat{G})\rightarrow +\infty$. Then by \eqref{app-facets}, we deduce that
 \[\left|\varphi(\alpha t,\alpha x)-\left( -\alpha x\cdot \nu_i \cot\theta_i +\frac{c_f }{\sin\theta_i}\alpha t +\frac{\alpha \tau_i}{\sin\theta_i}\right)\right|\rightarrow 0\] and 
 \[|\nabla\varphi(\alpha t,\alpha x)+ \nu_i\cot \theta_i|\rightarrow 0.\] Hence
		\[\overline{\xi}(t,x,y)\rightarrow x\cdot \nu_i \cos\theta_i +y\sin\theta_i-c_f t + \tau_i,\]
		and by \eqref{ibju}, $|\varphi(\alpha t,\alpha x)-\alpha y|$ is bounded. By \eqref{pibpju}, 
		\[q_i(\alpha t,\alpha x,\varphi(\alpha t,\alpha x)) \hbox{ is bounded and } q_j(\alpha t,\alpha x,\varphi(\alpha t,\alpha x)) \rightarrow +\infty  \hbox{ for $j\neq i$}.\]
		It follows that $h(\alpha t,\alpha x)\rightarrow 0$, which gives us \eqref{coincide at infinity}.
	\end{proof}
	\vskip 0.3cm

 This lemma has the following corollary.
 \begin{corollary}\label{coro:rhoe}
 There exists an $\rho(\varepsilon)$ such that
 \[|\overline{u}(t,x,y)-\underline{u}(t,x,y)|\le (n^2+1)\varepsilon, \quad \hbox{in $\{d((t,x,y),\mathcal{R})\ge \rho(\varepsilon)\}$}.\]
 \end{corollary}

\subsection{Existence of entire solutions}

In this subsection, we prove the following result for \eqref{RDxy}, which clearly implies Theorem~\ref{Th1}.

\begin{proposition}\label{Proposition1}
Suppose $\{(\nu_i,\theta_i,\tau_i)\}_{i=1,\cdots,n}$ of $\mathbb{S}^{N-2}\times(0,\pi/2]\times\R$ satisfies $(\nu_i,\theta_i)\neq (\nu_j,\theta_j)$ for $i\neq j$. Then there exists an entire solution $U_{\nu_i,\theta_i,\tau_i}$ of \eqref{RDxy} satisfying
\[\underline{u}<U_{\nu_i,\theta_i,\tau_i}<1 \hbox{ in $\R\times\R^N$},\]
and
\[
\left|U_{\nu_i,\theta_i,\tau_i}(t,x,y)-\underline{u}(t,x,y)\right|\rightarrow 0,
\]
uniformly as $d((t,x,y),\mathcal{R})\rightarrow +\infty$. Moreover,
\begin{itemize}
\item for any $(t_0,x_0,y_0)\in\R\times\R^N$ satisfying $\min_{i=\{1,\cdots,n\}}\{x_0\cdot \nu_i \cos\theta_i +y_0\sin\theta_i-c_f t_0\}\ge 0$,  we have
\be\label{cone of monotonicity}
U(t-t_0,x-x_0,y-y_0) \ge U(t,x,y),
\ee
where the inequality is strict if $\min_{i=\{1,\cdots,n\}}\{x_0\cdot \nu_i \cos\theta_i +y_0\sin\theta_i-c_f t_0\}> 0$;
\item for fixed $(\nu_i,\theta_i)$, $U_{\nu_i,\theta_i,\tau_i}$ are decreasing in $\tau_i\in\R$. 
\end{itemize}
\end{proposition}

\begin{proof}
Let $u_n$ be the solution of \eqref{RDxy} for $t\ge -n$ with initial data
$$u_n(-n,x,y)=\underline{u}(-n,x,y), \hbox{ for $(x,y)\in\R^N$}.$$
By Lemmas~\ref{lem supersolution}-\ref{lem comparison of sup and sub} and the comparison principle, 
\be\label{entire}
\underline{u}(t,x,y)\le u_n(t,x,y)\le \overline{u}_{\varepsilon,\alpha}(t,x,y), \hbox{ for $t>-n$ and $(x,y)\in\R^N$},
\ee
where $\varepsilon>0$, $0<\alpha<\alpha(\varepsilon)$ are small constants satisfying the assumptions in Lemma \ref{lem supersolution}.

Since $\underline{u}$ is a subsolution, the sequence $u_n$ is increasing in $n$. Letting $n\rightarrow +\infty$,
by parabolic estimates, the sequence $u_n$ converges locally uniformly in $C^{1,2}$ to an entire solution $U$ of \eqref{RDxy}. By \eqref{entire} and the strong maximum principle,
\be\label{u-U-u}
\underline{u}<U<\overline{u}_{\varepsilon,\alpha}  \hbox{ in } \R\times \R^N.
\ee
Furthermore, by Corollary~\ref{coro:rhoe}, $U$ satisfies
\[|U(t,x,y)-\underline{u}(t,x,y)|\le (n^2+1)\varepsilon \hbox{ in }\{d((t,x,y),\mathcal{R})\ge \rho(\varepsilon)\}.\]
Since $\varepsilon>0$ is arbitrary, this implies that
\[|U(t,x,y)-\underline{u}(t,x,y)|\rightarrow 0 \hbox{ uniformly as } d((t,x,y),\mathcal{R})\rightarrow +\infty.\]

Since $g^\prime<0$, for any $(t_0,x_0,y_0)\in\R\times\R^N$ satisfying $\min_{i=\{1,\cdots,n\}}\{x_0\cdot \nu_i \cos\theta_i +y_0\sin\theta_i-c_f t_0\}\ge 0$,  we have
\[\underline{u}(t-t_0,x-x_0,y-y_0)=\max_{1\le i\le n}\{g((x-x_0)\cdot \nu_i \cos\theta_i +(y-y_0)\sin\theta_i-c_f (t-t_0) + \tau_i)\ge \underline{u}(t,x,y),\]
where the inequality is strict if $\min_{i=\{1,\cdots,n\}}\{x_0\cdot \nu_i \cos\theta_i +y_0\sin\theta_i-c_f t_0\}> 0$. Then by the comparison principle, $u_n(t-t_0,x-x_0,y-y_0)\ge u_n(t,x,y)$. Letting $n\rightarrow +\infty$, we also get $U(t-t_0,x-x_0,y-y_0)\ge U(t,x,y)$. By the asymptotic behavior of $U$ and using the comparison principle once again, we deduce that $U(t-t_0,x-x_0,y-y_0)> U(t,x,y)$ if $\min_{i=\{1,\cdots,n\}}\{x_0\cdot \nu_i \cos\theta_i +y_0\sin\theta_i-c_f t_0\}> 0$.  

The monotonicity of $U_{\nu_i,\theta_i,\tau_i}$ with respect to $\tau_i$ can be got in the same way.
\end{proof}

\subsection{Uniqueness}
Let $U$ denote $U_{\nu_i,\theta_i,\tau_i}$ in Proposition~\ref{Proposition1} for short.  It is unique in the following sense.

\begin{proposition}\label{uniqueness}
If  $v$  is an entire solution of \eqref{RDxy} satisfying 
\be\label{asy-v}
|v(t,x,y)-\underline{u}(t,x,y)|\rightarrow 0 \hbox{ uniformly as $d((t,x,y),\mathcal{R})\rightarrow +\infty$},
\ee
then $v\equiv U$ in $\R\times\R^N$.
\end{proposition}

\begin{proof}
Since $v$ satisfies \eqref{asy-v}, it is a transition front connecting $0$ and $1$ with sets $\Gamma_t$, $\Omega_t^{\pm}$ defined by \eqref{definition of Gamma} and \eqref{definition of Omega}. Hence there exists an $R>0$ such that 
$$0<v(t,x,y)\le \sigma \hbox{ for $(t,x,y)\in \omega^-$} \hbox{ and }1-\sigma\le v(t,x,y)<1 \hbox{ for $(t,x,y)\in \omega^+$},$$
where
\[\omega^{\pm}:=\left\{(t,x,y)\in\R\times\R^N; \min_{1\le i\le n} \{x\cdot \nu_i \cos\theta_i +y\sin\theta_i -c_f t +\tau_i\}\le -R, (\text{or}\ \ge R)\right\}.\]

Let
\[\omega:=\left\{(t,x,y)\in\R\times\R^N; -R\le \min_{1\le i\le n} \{x\cdot \nu_i \cos\theta_i +y\sin\theta_i -c_f t +\tau_i\}\le R\right\}.\]

Now we compare $v$ and $U$ by the sliding method. This is almost the same with the proof of \cite[Theorem~1.12]{BH2}, so we only  sketch  the proof.

{\bf Step 1. For $\tau$ large enough,  $ U(t+\tau,x,y)\ge v(t,x,y) \hbox{ in } \R\times\R^N$.}

Since $U$ is a transition front connecting $0$ and $1$ with sets $\Gamma_t$ and $\Omega_t^{\pm}$, one has that
\begin{itemize}
    \item $U(t+\tau,x,y)\rightarrow 1$ uniformly as $\tau\rightarrow +\infty$ for $(t,x,y)\in\omega$;
    \item $U(t+\tau,x,y)\rightarrow 0$ uniformly as $\tau\rightarrow -\infty$ for $(t,x,y)\in\omega$.
\end{itemize} 
This implies that 
\[U(t+\tau,x,y)\ge v(t,x,y) \hbox{ in $\omega$ for large $\tau$}.\]
In particular, $U(t+\tau,x,y)\ge v(t,x,y)$ on $\partial \omega^{\pm}$ for large $\tau$.

We next show that $U(t+\tau,x,y)\ge v(t,x,y)$ in $\omega^{\pm}$ for large $\tau$. Define
\[\varepsilon_\ast=\inf\left\{\varepsilon>0; U(t+\tau,x,y)\ge v(t,x,y)-\varepsilon \hbox{ in } \omega^-\right\}.\]
Since $0<U,\ v<1$, $\varepsilon_\ast$ is well-defined. In fact, by the definition of $\omega^\pm$,  $\varepsilon_\ast\le \sigma$. Assume  $\varepsilon_\ast>0$. Then
\[U(t+\tau,x,y)>v(t,x,y)-\varepsilon_\ast \hbox{ on } \partial\omega^-.\]
Since $v\le \sigma$ in $\omega^-$ and $f$ is decreasing in $(-\infty,2\sigma]$, 
\[\partial_t(v-\varepsilon_\ast) -\Delta (v-\varepsilon_\ast)\le f(v-\varepsilon_\ast).\]
Let $z(t,x,y):=U(t+\tau,x,y)-v(t,x,y)+\varepsilon_\ast$. It is a nonegative function and satisfies $\partial_tz-\Delta z\ge bz$ in $\omega^-$ for some bounded function $b$. By the definition of $\varepsilon_\ast$, there exists a sequence $\{(t_n,x_n,y_n)\}_{n\in\mathbb{N}}$ in $\omega^-$ such that 
\be\label{2.4.1}
U(t_n+\tau,x_n,y_n)-v(t_n,x_n,y_n)+\varepsilon_\ast\rightarrow 0 \hbox{ as } n\rightarrow +\infty.
\ee
Notice that $d((t_n,x_n,y_n),\partial\omega^-)$ must be bounded. This is because otherwise we would have $U(t_n+\tau,x_n,y_n)\rightarrow 0$ and $v(t_n,x_n,y_n)\rightarrow 0$, and \eqref{2.4.1} can not be satisfied. 

Then there exists a sequence $\{x_n',y_n'\}$ such that $(t_n-1,x_n^\prime,y_n^\prime)\in\partial\omega^-$ and $|(x_n,y_n)-(x_n',y_n')|<+\infty$. By linear parabolic estimates, $z(t_n,x_n,y_n)\rightarrow 0$ implies $z(t_n-1,x_n^\prime,y_n^\prime)\rightarrow 0$. This is impossible since $U(t+\tau,x,y)\ge v(t,x,y)$ on $\partial\omega^-$ and $\varepsilon_\ast>0$.

In conclusion, $\varepsilon_\ast=0$ and $U(t+\tau,x,y)\ge v(t,x,y)$ in $\omega^-$. Similarly, one can prove that $U(t+\tau,x,y)\ge v(t,x,y)$ in $\omega^+$. Thus $U(t+\tau,x,y)\ge v(t,x,y)$ in $\R\times\R^N$ for all  large $\tau$.

{\bf Step 2.} Define
\[\tau_\ast=\inf\left\{\tau\in\R; U(t+\tau,x,y)\ge v(t,x,y) \hbox{ in } \R\times\R^N\right\}.\]
 Since $U$, $v$ satisfy \eqref{asy-v} and $\underline{u}(t+\tau,x,y)>\underline{u}(t,x,y)$ for $\tau>0$, $\tau_\ast$ is well-defined and $0\le \tau_\ast<+\infty$. We need to show that $\tau_\ast=0$.

Assume by the contrary that $\tau_\ast>0$. Then two cases may occur: either
\be\label{case1}
\inf_{\omega}\left\{U(t+\tau_\ast,x,y)-v(t,x,y)\right\}>0
\ee
or
\be\label{case2}
\inf_{\omega}\left\{U(t+\tau_\ast,x,y)-v(t,x,y)\right\}=0
\ee

In the case of \eqref{case1},  there exists an $\eta_0>0$ such that for any $\eta\in (0,\eta_0]$, $U(t+\tau_\ast-\eta,x,y)-v(t,x,y)\ge 0$ in $\omega$. By the same argument in Step 1, one can prove that $U(t+\tau_\ast-\eta,x,y)-v(t,x,y)\ge 0$ in $\omega^{\pm}$, which contradicts the definition of $\tau_\ast$. 

In the case of \eqref{case2}, there is a sequence $\{(t_k,x_k,y_k)\}_{k\in\mathbb{N}}\in\omega$ such that 
\be\label{4.01}
\lim_{n\rightarrow +\infty}\left[U(t_k+\tau_\ast,x_k,y_k)-v(t_k,x_k,y_k)\right]=0.
\ee
If $d((t_k,x_k,y_k),\mathcal{R})\rightarrow +\infty$, then $U(t_k+\tau_\ast,x_k,y_k)-\underline{u}(t_k+\tau_\ast,x_k,y_k)\rightarrow 0$, $v(t_k,x_k,y_k)- \underline{u}(t_k,x_k,y_k)\rightarrow 0$. But since $\underline{u}(t_k+\tau,x_k,y_k)>\underline{u}(t_k,x_k,y_k)+\alpha(\tau)$ for $\tau>0$, where $\alpha$ is a modulus function, we get a contradiction with \eqref{4.01}. In other words, we must have
\be\label{4.02}
\limsup_{k\to\infty}d((t_k,x_k,y_k),\mathcal{R})<+\infty.
\ee
Assume without loss of generality that $(t_k,x_k,y_k)\in \mathcal{R}$. Take a large $r_1>0$ (to be determined below). Then
\[\min_{1\le i\le n}\{(x_k\cdot \nu_i \cos\theta_i +y_k \sin\theta_i+r_1 \sin\theta_i -c_f t_k\}\ge r_1\min_{1\le i\le n} \sin\theta_i.\]
In the left hand side of this inequality, assume the minima is attained at $i$ and let the minimum be $r_2$. Then $r_2\ge r_1\min_{1\le i\le n} \sin\theta_i$. Let $(x_k^\prime,y_k^\prime)=(x_k,y_k+r_1-c_f/\sin\theta_i) -r_2 (\nu_i\cos\theta_i,\sin\theta_i)$. Then
\[x_k^\prime\cdot \nu_i\cos\theta_i+ y_k^\prime \sin\theta_i -c_f (t_k-1)=0,\]
and for any $j\neq i$,
\[x_k^\prime\cdot \nu_j\cos\theta_j+ y_k^\prime \sin\theta_j -c_f (t_k-1)\ge 
r_1\min_{1\le \ell\le n} \sin\theta_\ell+c_f \left(1-\frac{\sin\theta_j}{\sin\theta_i}\right).\]
If $r_1$ is sufficiently large, these two relations imply that $(t_k-1,x_k^\prime,y_k^\prime)\in \omega$  and $d((t_k-1,x_k^\prime,y_k^\prime),\mathcal{R})$ are  very large. Then  as in the proof of \eqref{4.02}, we deduce that $U(t_k-1+\tau_\ast,x_k^\prime,y_k^\prime)-v(t_k-1,x_k^\prime,y_k^\prime)\ge \alpha(\tau_\ast)$ for some $\alpha(\tau_\ast)>0$. On the other hand, by the definition of $(x_k^\prime,y_k^\prime)$, we also have
\[\limsup_{k\to+\infty}|(x_k,y_k)-(x_k^\prime,y_k^\prime)|<+\infty.\] Then  linear parabolic estimates imply $U(t_k-1+\tau_\ast,x_k^\prime,y_k^\prime)-v(t_k-1,x_k^\prime,y_k^\prime)\rightarrow 0$, which is a contradiction.
\end{proof}

\begin{proposition}\label{pro continuity}
 For fixed $(\nu_i,\theta_i)$, $U_{\nu_i,\theta_i,\tau_i}(t,x,y)$ depend continuously on $(\tau_1,\cdots,\tau_n)\in\R^n$ in the sense of $\mathcal{T}$.
\end{proposition}

\begin{proof}
Let us denote $\underline{u}$ and $\overline{u}$ by $\underline{u}_{\tau_i}$ and $\overline{u}_{\tau_i}$ respectively, $\varphi$ in Section~\ref{subsection:phi} by $\varphi_{\tau_i}$. By the continuity of $g$ and the definition of $\underline{u}_{\tau_i}$,  $\underline{u}_{\tau_i}$ is continuous in $\tau_i$. By the definition of $\varphi_{\tau_i}$, $\varphi_{\tau_i}$ depends continuously on $\tau_i$. The continuous dependence of $\overline{u}_{\tau_i}$ on $\tau_i$ follows from its definition and the continuity of $g$. Take any sequence $\tau_i^k$ for every $i\in\{1,\cdots,n\}$ such that $\tau_i^k\rightarrow \tau_i$ as $k\rightarrow +\infty$. By parabolic estimates, $U_{\nu_i,\theta_i,\tau_i^k}$ (up to a subsequence) converges locally uniformly in $C^{1,2}$ to an entire solution $u_{\infty}$ of \eqref{RDxy}. By continuity of $\underline{u}_{\tau_i}$ and $\overline{u}_{\tau_i}$ with respect to $\tau_i$,  $u_{\infty}$ satisfies \eqref{u-U-u}. Then by uniqueness of $U_{\nu_i,\theta_i,\tau_i}$,  $u_{\infty}\equiv U_{\nu_i,\theta_i,\tau_i}$. By monotonicity of $U_{\nu_i,\theta_i,\tau_i}$ with respect to $\tau_i$, $U_{\nu_i,\theta_i,\tau_i^k}$ converge to $U_{\nu_i,\theta_i,\tau_i}$ as $k\rightarrow +\infty$ in the sense of the topology $\mathcal{T}$.
\end{proof}
\vskip 0.3cm

Finally, we show that $U$ is  monotone in $t$.

\begin{lemma}\label{Ut}
In $\R\times\R^N$, $\partial_t U>0$. Moreover, for any $\rho>0$, there exists a $k(\rho)>0$ such that
\[\partial_t U\ge k(\rho) \hbox{ in }  \left\{(t,x,y): d((t,x,y),\partial \mathcal{P})\le \rho\right\}.\]
\end{lemma}

\begin{proof}
One can easily check that $U(t,x,y)$ is a transition front connecting $0$ and $1$ with sets
$\Gamma_t$
and
$\Omega_t^{\pm}$
for every $t\in\R$. Then by \cite{GH}, one knows that $\partial_t U>0$ in $\R\times \R^N$.

The monotonicity of $U$ in $t$ can also be seen from its construction. In the proof of Proposition~\ref{Proposition1}, by noticing that $\underline{u}$ is increasing in $t$,  the maximum principle implies  that $\partial_tu_n(t,x,y)>0$ for all $t>-n$ and $(x,y)\in\R^N$. By first letting $n\rightarrow +\infty$ and then applying the strong maximum principle to the linear parabolic equation $\partial_t \partial_tU-\Delta \partial_tU =f^\prime(U)\partial_tU$, one concludes that $\partial_tU>0$ in $\R\times\R^N$.

Assume that there is a constant $\rho>0$ and a sequence $\{(t_k,x_k,y_k)\}_{k\in\mathbb{N}}$ of $\R\times\R^N$ such that 
\be\label{utn}
d((t_k,x_k,y_k),\partial\mathcal{P})\le \rho \hbox{ and } \partial_t U(t_k,x_k,y_k)\rightarrow 0 \hbox{ as $k\rightarrow +\infty$}. 
\ee
That is, $\min_{1\le i\le n} \{x_k\cdot \nu_i \cos\theta_i +y_k\sin\theta_i -c_f t_k +\tau_i\}\le \rho$. Let $u_k(t,x,y)=U(t+t_k,x+x_k,y+y_k)$. Then  one has  
\begin{equation}\label{3.1}
u_k(t,0,0)\rightarrow 1,\ 0 \hbox{ as $t\rightarrow +\infty$ and $-\infty$ respectively},
\end{equation}
for all $k\in\mathbb{N}$. By standard parabolic estimates, $u_k(t,x,y)$ converge, up to  a subsequence, locally uniformly to a  solution $u_{\infty}(t,x,y)$ of \eqref{RDxy}. By \eqref{utn}, we have
\[\partial_tu_{\infty}(0,0,0)=0.\]
Because  $\partial_tu_{\infty}\ge 0$ in $\R\times\R^N$, applying the strong maximum principle to $\partial_tu_{\infty}$ shows that $\partial_tu_{\infty}\equiv 0$ for $t\le 0$. But there is impossible by \eqref{3.1}
and the fact that $0< u_{\infty}(0,0,0)<1$ for $(t,x,y)\in\R\times\R^N$.
\end{proof}

\section{Stability of the entire solution}

In this section, we prove Theorem~\ref{Th2}, that is, stability of the entire solution. 

We still consider equation~\eqref{RDxy} with $e_0=(0,\cdots,1)$. Let $U$ denote $U_{\nu_i,\theta_i,\tau_i}$ in Proposition~\ref{Proposition1} for short. We assume that $u$ is a solution to the Cauchy problem of \eqref{RDxy},
\begin{eqnarray}\label{CP}
\left\{\begin{array}{lll}
\partial_tu-\Delta_x u-u_{yy}=f(u), && t>0,\ (x,y)\in\R^N,\\
u(0,x,y)=u_0(x,y), && (x,y)\in\R^N.
\end{array}
\right.
\end{eqnarray}
The initial value $u_0$ satisfies $0\le u_0(x,y)\le 1$ and
\be\label{initial}
|u_0(x,y)-\underline{u}(0,x,y)|\rightarrow 0, \hbox{uniformly as $d((x,y),\mathcal{R}_0)\rightarrow +\infty$},
\ee
where $\mathcal{R}_0$ is time slice of the ridges  $\mathcal{R}$ at $t=0$.

We first prove a technical tool which will be used later on and also in Section~5.  In the following we fix a $\delta\in (0,\sigma]$, where $\sigma$ is given by \eqref{sigma}.

\begin{lemma}\label{supsubsolution}
Assume that $u^\ast\ge 0$ is a $C^{1,2}(\R\times\R^N)$ function satisfying
\begin{itemize}
\item[(i)]  $k:=\inf_{(t,x)\in \{\delta\le u^\ast\le 1-\delta\}}\partial_tu^\ast(t,x)>0$;

\item[(ii)] in $\{u^\ast\le \delta\}$ or $\{u^\ast\ge 1-\delta\}$, $\partial_tu^\ast\ge -\frac{\mu k}{\mu +L}$;

\item[(iii)] in $\{0\le u^\ast<1\}$,
\be\label{super-ine}
\partial_tu^\ast-\Delta u^\ast -f(u^\ast)\ge 0,
\ee
\end{itemize}
Then there exists a $\omega>0$ such that
\[u^+(t,x):=\min\left\{u^\ast(t+\omega\delta(1-e^{-\mu t}),x)+\delta e^{-\mu t},1\right\} \]
is a supersolution  of \eqref{RD} in $\R^+\times\R^N$.

Similarly, if  $u_\ast\le 1$ is a $C^{1,2}(\R\times\R^N)$ function  satisfying (i), (ii) and
\begin{itemize}
\item[(iii)'] in $\{0< u^\ast\le 1\}$,
\be\label{sub-ine}
\partial_tu_{\ast}-\Delta u_\ast -f(u_\ast)\le 0.
\ee
\end{itemize}
Then there exists a $\omega>0$ such that
\[ u^-(t,x):=\max\left\{u_\ast(t-\omega\delta(1-e^{-\mu t}),x)-\delta e^{-\mu t},0\right\}\]
is a subsolution of \eqref{RD} in $\R^+\times\R^N$.
\end{lemma}
\begin{proof}
Let 
\[\omega=\frac{\mu +L}{\mu k}.\]

We first show the supersolution property. Since $1$ is a solution of \eqref{RD}, we need only to check \eqref{super-ine} in the open set $\left\{u^+<1\right\}$. Here
\begin{align*}
N:=&\partial_tu^+ -\Delta u^+ -f(u^+)\\
=&\partial_tu^\ast\left(1+\omega\delta \mu e^{-\mu t}\right) -\delta \mu e^{-\mu t} -\Delta u^\ast -f\left(u^+\right)\\
\ge& \omega\delta \mu \partial_tu^\ast e^{-\mu t}  -\delta \mu e^{-\mu t} +f\left(u^\ast\right)-f\left(u^\ast+\delta e^{-\mu t}\right).
\end{align*}
\begin{enumerate}
    \item In $\left\{\delta\le u^\ast\le 1-\delta\right\}$,  by the lower bound of $\partial_tu^\ast$ from (i) and Lagrange mean value theorem,
\[N\ge \omega k \delta \mu e^{-\mu t} - \delta \mu e^{-\mu t} -L \delta e^{-\mu t}\ge 0.\]
\item In $\left\{0\le u^\ast \le \delta\right\}$, by Lagrange mean value theorem and \eqref{sigma},
\[N\ge   -2\delta \mu e^{-\mu t} -\frac{f^\prime(0)}{2} \delta e^{-\mu t}\ge 0.\]
\item In $\left\{1-\delta\le u^\ast<1\right\}$, by Lagrange mean value theorem and \eqref{sigma},
\[ N\ge   -2\delta \mu e^{-\mu t} -\frac{f^\prime(1)}{2} \delta e^{-\mu t}\ge 0.\]
\end{enumerate}

We then show the subsolution property. Since $0$ is a solution of \eqref{RD}, we only need to check \eqref{sub-ine} in the open set $\left\{u^->0\right\}$. Here
\begin{align*}
N:=&\partial_tu^- -\Delta u^- -f\left(u^-\right)\\
=&\partial_tu_\ast\left(1-\omega\delta \mu e^{-\mu t}\right) +\delta \mu e^{-\mu t} -\Delta u_\ast -f\left(u^-\right)\\
\le& -\omega\delta \mu \partial_tu_\ast e^{-\mu t}  +\delta \mu e^{-\mu t} +f\left(u_\ast\right)-f\left(u_\ast-\delta e^{-\mu t}\right).
\end{align*}
\begin{enumerate}
    \item In $\left\{\delta\le u_\ast\le 1-\delta\right\}$, by the lower bound of $\partial_tu_\ast$ from (i) and Lagrange mean value theorem,
\[N\le -\omega k \delta \mu e^{-\mu t} + \delta \mu e^{-\mu t} +L \delta e^{-\mu t}\le 0.\]
\item In $\left\{0< u_\ast \le\delta\right\}$,  by Lagrange mean value theorem and \eqref{sigma},
\[ N\le   2\delta \mu e^{-\mu t} +\frac{f^\prime(0)}{2} \delta e^{-\mu t}\le 0 .\]
\item In $\left\{1-\delta\le u_\ast\le 1\right\}$,  by Lagrange mean value theorem and \eqref{sigma},
\[  N\le   2\delta \mu e^{-\mu t} +\frac{f^\prime(1)}{2} \delta e^{-\mu t}\le 0. \]
\end{enumerate}
\end{proof}

\subsection{A supersolution}

For $t\ge 0$ and $(x,y)\in\R^N$, define
\[u^+(t,x,y):= \min\left\{\overline{u}(t+\omega\delta (1 - e^{-\mu t}),x,y) +\delta e^{-\mu t},1 \right\},\]
where $\omega$ is a constant to be determined, $\overline{u}$ is the supersolution defined in Subsection \ref{subsection:sup-sub}. 

Before we show that $u^+$ is a supersolution, we need the following technical lemma.
\begin{lemma}\label{phi-alphat0}
For any $R>0$, in  $\{(x,y): d((x,y),\mathcal{R}_0)\le R\}$, 
\begin{equation}\label{phi4.1}
    y-\frac{1}{\alpha}\varphi(0,\alpha x)\rightarrow -\infty  ~~\hbox{uniformly  as $\alpha \rightarrow 0^+$}.
\end{equation}
\end{lemma}

\begin{proof}
Because $\mathcal{R}_0=\cup_{i,j}\mathcal{R}^{ij}_0$,  where $\mathcal{R}_0^{ij}=\widetilde{P}_{i,0}\cap\widetilde{P}_{j,0}$, we may assume that  $d((x,y),\mathcal{R}^{ij}_0)\le R$ for some $i,\ j\in\{1,\cdots,n\}$. This implies that 
\[x\cdot \nu_i \cos\theta_i +y\sin\theta_i +\tau_i \le R \hbox{ and } x\cdot \nu_j \cos\theta_j +y\sin\theta_j +\tau_j \le R.\]
Then
\begin{align*}
q_i(0,\alpha x,\varphi(0,\alpha x))=&\alpha x\cdot \nu_i \cos\theta_i +\varphi(0,\alpha x) \sin\theta_i +\alpha \tau_i\\
\le& \alpha R -\alpha \left(y-\frac{1}{\alpha}\varphi(0,\alpha x)\right)\sin\theta_i\\
\le& \alpha R +\alpha \left|y-\frac{1}{\alpha}\varphi(0,\alpha x)\right|.
\end{align*}
Similarly, 
\[q_j(0,\alpha x,\varphi(0,\alpha x))\le \alpha R +\alpha \left|y-\frac{1}{\alpha}\varphi(0,\alpha x)\right|.\]
Since $\sum_{i=1}^n e^{-q_i(0,\alpha x,\varphi(0,\alpha x))}=1$, $n\ge 2$ and $\sin\theta_i>0$, the above two inequalities imply that
\[1\ge 2 e^{-\left(\alpha R+\alpha |y-\frac{1}{\alpha}\varphi(0,\alpha x)|\right)}.\]
Hence
\[\left|y-\frac{1}{\alpha}\varphi(0,\alpha x)\right|\ge \frac{1}{\alpha}\ln 2 -R,\]
which implies that either $\left|y-\frac{1}{\alpha}\varphi(0,\alpha x)\right|\rightarrow +\infty$ or $-\infty$ as $\alpha \rightarrow 0^+$. Since $\varphi(0,\alpha x)>\alpha y$ for $(x,y)$ on $\widetilde{P}_{i,0}$ and $\widetilde{P}_{j,0}$,  only the first possibility, that is,
\eqref{phi4.1} is possible.
\end{proof}
\vskip 0.3cm

Now we show that $u^+$ is a supersolution of \eqref{CP}. 

\begin{lemma}\label{u+}
There exist $\omega>0$ and $\mu>0$ such that for any $\alpha>0$ small enough and $\delta\in (0,\sigma]$, $u^+$ is a supersolution of \eqref{CP}. Furthermore,
\be\label{comparison from above}
u^+(t,x,y)\ge u(t,x,y), \hbox{ for $t\ge 0$ and $(x,y)\in\R^N$}.
\ee
\end{lemma}

\begin{proof}
By the definition of $u^+$,  we need only to restrict our attention to the domain where $\overline{u}(t,x,y)=g(\overline{\xi}(t,x,y))+\varepsilon h(\alpha t,\alpha x)$. By Lemma~\ref{pro-xi},
\[\partial_t\overline{u}=g^\prime(\overline{\xi})\partial_t\overline{\xi} +\varepsilon \alpha \partial_th(\alpha t,\alpha x)\ge -c_f g^\prime(\overline{\xi})-\alpha Ch(\alpha t,\alpha x) .\]
Because $g^\prime\leq 0$ and $-g^\prime\ge k$ in $[-R,R]$, if $\alpha$ is small enough, then $\overline{u}$ satisfies the assumptions (i)-(ii) in Lemma~\ref{supsubsolution}. Furthermore, by Lemma~\ref{lem supersolution}, (iii) of Lemma~\ref{supsubsolution} is also satisfied. Then by Lemma \ref{supsubsolution}, we deduce that $u^+$ is a supersolution of \eqref{CP}. 

Next let us check the initial value condition. Since $u(0,x,y)=u_0(x,y)$ satisfies \eqref{initial}, there is $R>0$ such that
 \[|u_0(x,y)-\underline{u}(0,x,y)|\le \delta, \hbox{ if $d((x,y),\mathcal{R}_0)\ge R$}.\]
Thus if  $d((x,y),\mathcal{R}_0)\ge R$, then
\[u^+(0,x,y)=\min\{\overline{u}(0,x,y)+\delta,1\}\ge \min\{\underline{u}(0,x,y)+\delta,1\}\ge u_0(x,y).\]
By Lemma~\ref{phi-alphat0}, one can take $\alpha$ small enough so that for $d((x,y),\mathcal{R}_0)\le R$,
\[\overline{u}(0,x,y)\ge 1-\delta.\]
Then $u^+(0,x,y)=\min\{\overline{u}(0,x,y)+\delta,1\}= 1\ge u_0(x,y)$ in $\{d((x,y),\mathcal{R}_0)\le R\}$. Putting these two facts together, we see $u^+(0,x,y)\ge u_0(x,y)$ for all $(x,y)\in\R^N$.

Now $u^+$ is a super-solution and $u^+(0,x,y)\geq u(0,x,y)$, \eqref{comparison from above} follows from the comparison principle.
\end{proof}
\vskip 0.3cm

\subsection{Subsolutions}
In this subsection we construct subsolutions. For every $i\in \{1,\cdots,n\}$, consider the facet $\widetilde{Q}_{i}$ of the polytope $\mathcal{Q}$. For $j\in \{1,\cdots,n\}\setminus\{i\}$, let $G_{ij}=\widetilde{Q}_{i}\cap \widetilde{Q}_{j}$ be the ridges on the facet $\widetilde{Q}_{i}$ even if some $G_{ij}$ may be empty.  Then  for each $j\in \{1,\cdots,n\}\setminus \{i\}$, take two positive constants $\gamma_{ij}$, $\lambda_{ij}$ such that 
\begin{enumerate}
    \item $\gamma_{ij}+\lambda_{ij}=1$;
    \item $\lambda_{ij}$ is small enough so that 
\[\gamma_{ij}\sin\theta_i -\lambda_{ij}\sin\theta_j>0;\]
and for any $l\in \{1,\cdots,n\}\setminus \{i\}$,
\be\label{lambdaij}
\gamma_{ij}(1-\sin\theta_i\sin\theta_l -\nu_i\cdot \nu_l \cos\theta_i\cos\theta_l)-\lambda_{ij}>0.
\ee
\end{enumerate}
Let also $\gamma_{ii}=1$ and $\lambda_{ii}=0$ so that $\gamma_{ij}$ and $\lambda_{ij}$ are defined for every $j\in\{1,\cdots,n\}$. Let 
\[\nu_{ij}=\gamma_{ij} (-\nu_i\cos\theta_i,-\sin\theta_i)+\lambda_{ij}(\nu_j\cos\theta_j,\sin\theta_j).\]
Notice that $\nu_{ii}=(-\nu_i\cos\theta_i,-\sin\theta_i)=-e_i$ and  for any $j\in\{1,\cdots,n\}$, 
\[\nu_{ij}\cdot (0,\cdots,-1)=\gamma_{ij}\sin\theta_i -\lambda_{ij}\sin\theta_j>0 .\]

For every $j\in \{1,\cdots,n\}$, define the hyperplane
\[Q_{ij}=\left\{(t,x,y)\in\R\times\R^N;\  -\gamma_{ij}q_i(t,x,y) +\lambda_{ij} q_j(t,x,y)=0\right\}.\]
It is  a rotation of $Q_i$ against $G_{ij}$.
Let
\[q_{ij}(t,x,y):= -\gamma_{ij}q_i(t,x,y) +\lambda_{ij}q_j(t,x,y).\]
Notice that $Q_{ii}=Q_{i}$. Let $\mathcal{Q}_i$ be the unbounded polytope enclosed by $Q_{ij}$, that is,
\[\mathcal{Q}_i=\left\{(t,x,y)\in\R\times\R^N;\ \min_{1\le j\le n}\{q_{ij}(t,x,y)\}\ge 0\right\}.\]
The unboundedness property is guaranteed by $\nu_{ij}\cdot (0,\cdots,-1)>0$. 
\begin{lemma}
 For each $i$, $\widetilde{Q}_i=Q_{ii}\cap \partial \mathcal{Q}_i$, that is, $\mathcal{Q}$ and $\mathcal{Q}_i$ share the   facet $\widetilde{Q}_i$.
\end{lemma}
\begin{proof}
Let $\widetilde{Q}_{ii}$ be the facet $Q_{ii}\cap \partial \mathcal{Q}_i$, that is,
\[\widetilde{Q}_{ii}=\left\{(t,x,y)\in \R\times\R^N;\,\min_{1\le j\le n}\{q_{ij}(t,x,y)\}=q_{ii}(t,x,y)=0\right\}.\]
By the definition of $q_{ii}$, one knows that $q_{ii}(t,x,y)=-q_{i}(t,x,y)=0$ for $(t,x,y)\in \widetilde{Q}_{ii}$. Then  for $(t,x,y)\in\widetilde{Q}_{ii}$ and $j\neq i$, 
\[q_{ij}(t,x,y)=-\gamma_{ij} q_i(t,x,y) +\lambda_{ij} q_j(t,x,y)=\lambda_{ij} q_j(t,x,y)\ge 0.\]
This means
\[\widetilde{Q}_{ii}=\{(t,x,y)\in\R\times\R^N;\, \min_{1\le j\le n} \{q_j(t,x,y)\}=q_i(t,x,y)=0\}=Q_i\cap \partial\mathcal{Q}=\widetilde{Q}_i.\]
That is, $\mathcal{Q}$ and $\mathcal{Q}_i$ share $\widetilde{Q}_i$, see Figure~1.
\end{proof}
\vskip 0.3cm

\begin{figure}[ht]
  \centering
  \includegraphics[scale=0.3]{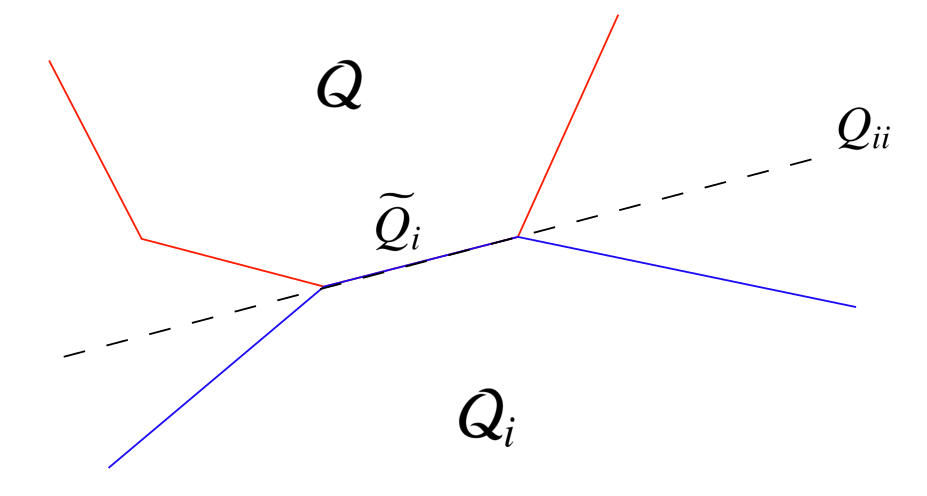} 
  \caption{Sketch of $\mathcal{Q}$ and $\mathcal{Q}_i$.}
\end{figure}

Let $\widehat{Q}_{ij}$ be the projection of $Q_{ij}$ on $(t,x)$-plane. Let $y=\psi_i(t,x)$ be the smooth function determined by 
\be\label{psi-i}
\sum_{j=1}^n e^{-q_{ij}(t,x,y)}=1.
\ee
As in  Subsection~\ref{subsection:phi},   there exists a constant $C>0$ such that for any $(t,x)\in \R\times\R^{N-1}$,
\[
\frac{1}{C}\sum_{\substack{k, l\in\{1,\cdots,n\}\\ k\neq l}} e^{-(\hat{q}_{ik} +\hat{q}_{il})}\le \left(-x\cdot \nu_i \cot\theta_i +\frac{c_f}{\sin\theta_i}t -\frac{\alpha\tau_i}{\sin\theta_i}\right)-\psi_i(t,x)\le C\sum_{\substack{k, l\in\{1,\cdots,n\} \\ k\neq l}} e^{-(\hat{q}_{ik} +\hat{q}_{il})},
\]
where $\hat{q}_{ij}$ represents $q_{ij}(t,x,\psi_i(t,x))$ for short.
As analysis in Subsection~\ref{subsection:phi}, one can show that $\psi_i$ satisfies the following properties.

\begin{lemma}\label{pro-psi}
There is $C>0$ independent of $\alpha$ such that for each $j\in J\cup\{i\}$, for any $(t,x)\in \widehat{Q}_{ij}$,
\be\label{psi-D1}
\left|\partial_t\psi_i(t,x)-\frac{\gamma_{ij} -\lambda_{ij}}{\gamma_{ij}\sin\theta_i -\lambda_{ij}\sin\theta_j}c_f\right|,\ \left|\nabla\psi_i +\frac{\gamma_{ij} \nu_i\cos\theta_i -\lambda_{ij} \nu_j\cos\theta_j}{\gamma_{ij} \sin\theta_i -\lambda_{ij}\sin\theta_j}\right|\le C\sum_{\substack{k, l\in\{1,\cdots,n\}\\ k\neq\ell}} e^{-(\hat{q}_{ik} +\hat{q}_{il})},
\ee
and any for $(t,x)\in\R\times\R^{N-1}$,
\[|\nabla\partial_t\psi_i(t,x)|,\ |\nabla^2\psi_i(t,x)|,\ |\nabla^3\psi_i(t,x)|\le C\sum_{\substack{k, l\in\{1,\cdots,n\}\\ k\neq\ell}} e^{-(\hat{q}_{ik} +\hat{q}_{il})}.\]
\end{lemma}

\begin{proof}
The proof is similar to that of Lemma~\ref{pro-phi}, but because now $q_{ij}$ are not distance functions, the computation needs some changes. We only show \eqref{psi-D1} and analogous computations can be applied to the rest. By substituting $y=\psi_i(t,x)$ into \eqref{psi-i} and differentiating it, one gets 
\be\label{psi-tx}
 \left\{\begin{aligned}
			&\partial_t\psi_i(t,x)=\frac{\sum_{j=1}^n e^{-\hat{q}_{ij}}(\gamma_{ij} -\lambda_{ij})c_f}{\sum_{j=1}^n e^{-\hat{q}_{ij}}(\gamma_{ij}\sin\theta_i -\lambda_{ij}\sin\theta_j)},\\
			& \nabla\psi_i(t,x)=-\frac{\sum_{j=1}^n e^{-\hat{q}_{ij}}(\gamma_{ij}\nu_i\cos\theta_i -\lambda_{ij}\nu_j\cos\theta_j)}{\sum_{j=1}^n e^{-\hat{q}_{ij}}(\gamma_{ij}\sin\theta_i -\lambda_{ij}\sin\theta_j)}.
		\end{aligned}\right.
\ee
Recall that $\lambda_{ij}$ is small so that $\gamma_{ij}\sin\theta_i -\lambda_{ij}\sin\theta_j>0$. Hence there is $C>0$ such that  for $(t,x)\in\widehat{Q}_{ij}$,
\begin{align*}
\left| \partial_t\psi_i - \frac{\gamma_{ij} -\lambda_{ij}}{\gamma_{ij}\sin\theta_i -\lambda_{ij}\sin\theta_j}c_f \right|
=&\left| \frac{\sum_{k=1}^n e^{-\hat{q}_{ik}} R_{ijk}}{\sum_{k=1}^n e^{-\hat{q}_{ik}}(\gamma_{ik}\sin\theta_i -\lambda_{ik}\sin\theta_k)(\gamma_{ij} \sin\theta_i -\lambda_{ij}\sin\theta_j)} \right|\\
\le& C\sum_{\substack{k, l\in\{1,\cdots,n\}\\ k\neq\ell}} e^{-(\hat{q}_{ik} +\hat{q}_{il})} ,
\end{align*}
and
\begin{align*}
\left| \nabla\psi_i + \frac{\gamma_{ij} \nu_i\cos\theta_i -\lambda_{ij} \nu_j\cos\theta_j}{\gamma_{ij} \sin\theta_i -\lambda_{ij}\sin\theta_j} \right|=&\left| \frac{\sum_{k=1}^n e^{-\hat{q}_{ik}} \hat{R}_{ijk}}{\sum_{k=1}^n e^{-\hat{q}_{ik}}(\gamma_{ik}\sin\theta_i -\lambda_{ik}\sin\theta_k)(\gamma_{ij} \sin\theta_i -\lambda_{ij}\sin\theta_j)} \right|\\
\le& C\sum_{\substack{k, l\in\{1,\cdots,n\}\\ k\neq\ell}}e^{-(\hat{q}_{ik} +\hat{q}_{il})} ,
\end{align*}
where 
\[R_{ijk}=(\gamma_{ik}-\lambda_{ik})(\gamma_{ij}\sin\theta_i -\lambda_{ij}\sin\theta_j)-(\gamma_{ij}-\lambda_{ij})(\gamma_{ik}\sin\theta_i -\lambda_{ik}\sin\theta_k),\]
and
\begin{align*}
\hat{R}_{ijk}=&(\gamma_{ij}\nu_i\cos\theta_i -\lambda_{ij}\nu_j\cos\theta_j)(\gamma_{ik} \sin\theta_i -\lambda_{ik}\sin\theta_k)\\
&-(\gamma_{ik}\nu_i\cos\theta_i -\lambda_{ik}\nu_k\cos\theta_k)(\gamma_{ij} \sin\theta_i -\lambda_{ij}\sin\theta_j).
\end{align*}
In the above, we have used the facts that $R_{ijk}$, $\hat{R}_{ijk}$ are bounded and $R_{ijk}=\hat{R}_{ijk}=0$ when $k=j$.
\end{proof}
\vskip 0.3cm

Fix any $i\in\{1,\cdots,n\}$. For $\varepsilon>0$, $\alpha>0$ and $(t,x,y)\in\R\times\R^N$, define 
\[u_{i, \ast}(t,x,y):=\max\left\{g(\xi_{i,\ast}(t,x,y)) -\varepsilon \hat{h}_i(\alpha t,\alpha x),0\right\},\]
where
\[\xi_{i,\ast}(t,x,y)=\frac{y -\frac{1}{\alpha}\psi_i(\alpha t,\alpha x)}{\sqrt{1+|\nabla\psi_i(\alpha t,\alpha x)|^2}} ~\hbox{ and } ~ \hat{h}_i(t,x)=\sum_{k, l\in\{1,\cdots,n\},k\neq l} e^{-(\hat{q}_{ik} +\hat{q}_{il})}.\]
In the sequel, we will denote $\psi_i(\alpha t,\alpha x)$ and its derivatives by $\psi_i$, $\partial_t\psi_i$, $\nabla\psi_i$, $\cdots$ for short.
A direct computation shows that
\[-\partial_t\xi_{i,\ast} -c_f =\frac{\partial_t\psi_i}{\sqrt{1+|\nabla\psi_i|^2}}-c_f +\frac{\alpha \nabla\psi_i\cdot \nabla\partial_t\psi_i}{1+|\nabla\psi_i|^2}\xi_{i,\ast}.\]
By Lemma~\ref{pro-psi}, one has the following lemma.

\begin{lemma}\label{pro-xi-i}
There exists a positive constant $C$ such that for $(t,x,y)\in\R\times\R^N$,
\be\label{uxi-i}
\frac{1}{C} \hat{h}_i(\alpha t,\alpha x)\le c_f-\frac{\partial_t\psi_i}{\sqrt{1+|\nabla\psi_i|^2}}\le C \hat{h}_i(\alpha t,\alpha x),
\ee
and
\[
 \left\{\begin{aligned}
			& \left|\frac{\alpha \nabla\psi_i\cdot \nabla\partial_t\psi_i}{1+|\nabla\psi_i|^2}\xi_{i,\ast}\right|\le \alpha C \hat{h}_i(\alpha t,\alpha x)|\xi_{i,\ast}|,\\
			& \left|1-|\nabla_x\xi_{i,\ast}|^2 -|\partial_y \xi_{i,\ast}|^2\right|\le \alpha C \hat{h}_i(\alpha t,\alpha x)(|\xi_{i,\ast}|+\xi_{i,\ast}^2),\\
   &|\Delta_x \xi_{i,\ast}|\le \alpha C h(\alpha t,\alpha x)(1+|\xi_{i,\ast}|+ \alpha |\xi_{i,\ast}|).
   \end{aligned}\right.
\]
\end{lemma}

\begin{proof}
We only show \eqref{uxi-i}. The rest inequalities follows from Lemma~\ref{pro-psi}, just as in the proof of Lemma~\ref{pro-xi}. After some computation and by \eqref{psi-tx}, one gets that
\begin{align*}
&c_f-\frac{\partial_t\psi_i}{\sqrt{1+|\nabla\psi_i|^2}}\\
=&c_f\left(1-\frac{\sum_{j=1}^n e^{-\hat{q}_{ij}}(\gamma_{ij}-\lambda_{ij})}{\sqrt{(\sum_{j=1}^n e^{-\hat{q}_{ij}}(\gamma_{ij}\sin\theta_i-\lambda_{ij}\sin\theta_j))^2 +(\sum_{j=1}^n e^{-\hat{q}_{ij}}(\gamma_{ij}\nu_j\cos\theta_i-\lambda_{ij}\nu_j\cos\theta_j))^2}} \right).
\end{align*}
Further computation shows that
\begin{align*}
&\left(\sum_{j=1}^n e^{-\hat{q}_{ij}}(\gamma_{ij}\sin\theta_i-\lambda_{ij}\sin\theta_j)\right)^2 +\left(\sum_{j=1}^n e^{-\hat{q}_{ij}}(\gamma_{ij}\nu_j\cos\theta_i-\lambda_{ij}\nu_j\cos\theta_j)\right)^2\\
=&\left(\sum_{j=1}^n e^{-\hat{q}_{ij}}(\gamma_{ij}-\lambda_{ij})\right)^2 +2\sum_{j=1}^n e^{-\hat{q}_{ij}}\gamma_{ij}\lambda_{ij}\left(1-\sin\theta_i\sin\theta_j -\nu_i\cdot\nu_j \cos\theta_i\cos\theta_j\right)\\
&+ \sum_{k, l\in\{1,\cdots,n\},k\neq l} e^{-(\hat{q}_{ik} +\hat{q}_{il})}\Big[\gamma_{ik}\lambda_{il}(1-\sin\theta_i\sin\theta_l -\nu_i\cdot \nu_l \cos\theta_i\cos\theta_l) \\
&+\lambda_{ik}\gamma_{il}(1-\sin\theta_k\sin\theta_i -\nu_k\cdot \nu_i\cos\theta_k\cos\theta_i)+\lambda_{ik}\lambda_{il}(\sin\theta_k\sin\theta_l +\nu_k\cdot\nu_l\cos\theta_k\cos\theta_l -1)\Big].
\end{align*}
Then \eqref{uxi-i} follows from \eqref{lambdaij} and the facts that $\theta_i\in (0,\pi/2]$, $(\nu_i,\theta_i)\neq (\nu_j,\theta_j)$ for $i\neq j$.
\end{proof}
\vskip 0.3cm

Now we show that $u_{i, \ast}$ is a subsolution and describe its asymptotic properties.   Let $\overline{\mathcal{R}}$ be the set of ridges and $\overline{P}_i$ be the facet of the polytope $\overline{\mathcal{P}}$ defined by
\[ \overline{\mathcal{P}}:=\left\{(t,x,y)\in\R^N; \min_{1\le j\le n}\{-\gamma_{ij} p_i(t,x,y) +\lambda_{ij} p_j(t,x,y)\}\ge 0\right\},\]
and let
\[p_i(t,x,y)=x\cdot \nu_i\cos\theta_i+y\sin\theta_i-c_ft +\tau_i, \hbox{ for $i=1,\cdots,n$}.\]
Notice that $\overline{\mathcal{P}}$ is the no-scaling version of $\mathcal{Q}_i$.

\begin{lemma}\label{asy-ui}
There exist $\varepsilon_0$ and $\alpha(\varepsilon)$ such that for $0<\varepsilon<\varepsilon_0$ and $0<\alpha\le \alpha(\varepsilon)$, the function $u_{i, \ast}$ is a subsolution of \eqref{RDxy} in $\R\times \R^N$.  Moreover, it satisfies
\begin{itemize}
\item[(i)] for $t\in\R$ and $(x,y)\in\R^N$,
 \[u_{i, \ast}(t,x,y)<g(x\cdot \nu_i \cos\theta_i +y\sin\theta_i -c_f t +\tau_i);\]

\item[(ii)] there exists a $\rho(\varepsilon)>0$ such that if $\min_{1\le j\le n}\{-\gamma_{ij} p_i(t,x,y) +\lambda_{ij} p_j(t,x,y)\}\le -\rho(\varepsilon)$, then
\[ 
u_{i, \ast}(t,x,y)\le (n^2+1)\varepsilon;
\]
and if $\min_{1\le j\le n}\{-\gamma_{ij} p_i(t,x,y) +\lambda_{ij} p_j(t,x,y)\}\ge \rho(\varepsilon)$, then
\[
u_{i, \ast}(t,x,y)\ge 1-(n^2+1)\varepsilon;
\]

\item[(iii)]  for any $M>0$, in $\{d((t,x,y),\overline{P}_i)\le M\}$,
\[
|u_{i, \ast}(t,x,y)-g(x\cdot \nu_i \cos\theta_i +y\sin\theta_i -c_f t +\tau_i)|\rightarrow 0 \quad \text{uniformly as $d((t,x,y),\overline{\mathcal{R}})\rightarrow +\infty$}.
\]
\end{itemize}
\end{lemma}

\begin{proof}
Let parameters $\sigma$, $L$, $R$, $k$ and $M$ be   as in the proof of Lemma~\ref{lem supersolution}.
By the definition of $\hat{h}_i$, there is $C>0$ such that
\[	\left\{\begin{aligned}
			&|\partial_t\hat{h}_i(t,x)|=\Big|c_f \sum_{\substack{k, l\in\{1,\cdots,n\}\\ k\neq\ell}} e^{-(\hat{q}_{ik} +\hat{q}_{il})}(\gamma_{ik}-\lambda_{ik}+\gamma_{il}-\lambda_{il})\Big|\le 2 c_f \hat{h}_i(t,x),\\
& |\Delta \hat{h}_i(t,x)|\le C \hat{h}_i(t,x).
\end{aligned}
\right.
\]
After increasing $C$,  Lemma~\ref{pro-xi-i} still holds.

Notice that $\hat{h}_i(t,x)$ is bounded, that is, $\hat{h}_i(t,x)\le n^2$ precisely.  Let 
		\[\varepsilon_0=\min \left\{\frac{\sigma}{n^2},\frac{k }{2CL},1\right\},\]
and for $0<\varepsilon<\varepsilon_0$,
\[\alpha(\varepsilon)=\min\left\{1,-\frac{f^\prime(0)\varepsilon}{2(6C M + 2 c_f + C)},-\frac{f^\prime(1)\varepsilon}{2(6C M + 2 c_f + C)}\right\}.\]

We need only to perform the following calculation in  $\{u_{i, \ast}>0\}$. By the definition of $u_{i, \ast}$,  Lemma \ref{pro-xi-i} and \eqref{M}, we have
\begin{align*}
N:=&\partial_tu_{i, \ast} -\Delta \underline{u}_{i} -f(\underline{u}_{i})\\
=&g^\prime(\xi_{i,\ast})(\partial_t\xi_{i,\ast} +c_f) +g^{\prime\prime}(\xi_{i,\ast})(1-\nabla_x \xi_{i,\ast}^2-\partial_y\xi_{i,\ast}^2)-g^\prime(\xi_{i,\ast})\Delta_x\xi_{i,\ast} \\
&+f(g(\xi_{i,\ast})) -f(u_{i, \ast})-\varepsilon \alpha \hat{h}_{i,t}(\alpha t,\alpha x) +\varepsilon \alpha^2 \Delta \hat{h}_i(\alpha t,\alpha x)\\
\le& g^\prime(\xi_{i,\ast})\frac{1}{C} \hat{h}_i(\alpha t,\alpha x)+M \alpha C \hat{h}_i(\alpha t,\alpha x) +2M \alpha C \hat{h}_i(\alpha t,\alpha x) +3M\alpha C \hat{h}_i(\alpha t,\alpha x)\\
&+f(g(\xi_{i,\ast})) -f(u_{i, \ast})+2\varepsilon \alpha c_f \hat{h}_i(\alpha t,\alpha x) +\varepsilon \alpha^2 C \hat{h}_i(\alpha t,\alpha x)\\
=& g^\prime(\xi_{i,\ast})\frac{1}{C} \hat{h}_i(\alpha t,\alpha x)+\alpha  \hat{h}_i(\alpha t,\alpha x)(6C M +  2 c_f +   C)+f(g(\xi_{i,\ast})) -f(u_{i, \ast}).
\end{align*}

In $\{\xi_{i,\ast}\ge R\}$, because $0<\varepsilon<\varepsilon_0$,  $u_{i, \ast}(t,x,y)=g(\xi_{i,\ast})-\varepsilon \hat{h}_i(\alpha t,\alpha x)\le 2\sigma$. Then
\[f(g(\xi_{i,\ast})) -f(u_{i, \ast})\le \frac{f^\prime(0)}{2} \varepsilon \hat{h}_i(\alpha t,\alpha x).\]
Thus by $g^\prime<0$ and $0<\alpha\le \alpha(\varepsilon)$,   $N\le 0$. Similarly, we deduce that $N\le 0$ in $\{\xi_{i,\ast}\ge R\}$. 

In $\{-R\le \xi_{i,\ast}\le R\}$, $-g^\prime(\xi_{i,\ast})\ge k$ and 
\[f(g(\xi_{i,\ast})) -f(u_{i, \ast})\le L \varepsilon \hat{h}_i(\alpha t,\alpha x).\]
Therefore, because $0<\varepsilon<\varepsilon_0$ and $0<\alpha\le \alpha(\varepsilon)$, we get
\[N(t,x,y)\le -k \frac{1}{C} \hat{h}_i(\alpha t,\alpha x)+\alpha  \hat{h}_i(\alpha t,\alpha x)(6C M +  2c_f +   C)+L \varepsilon \hat{h}_i(\alpha t,\alpha x)\le 0.\]

In conclusion, $N\le 0$ in the whole  $\{u_{i, \ast}>0\}$. The remaining properties of $u_{i, \ast}$ can be proved as in the proof of Lemma~\ref{lem comparison of sup and sub}. 
\end{proof}
\vskip 0.3cm

For $t\ge 0$ and $(x,y)\in\R^N$, define
\[u_i^-(t,x,y)= \max\left\{u_{i, \ast}(t-\omega \delta( 1 - e^{-\mu t}),x,y) -\delta e^{-\mu t},0\right\},\]
where $\omega$ is a constant to be determined.
Similar to Lemma~\ref{phi-alphat0}, one can prove the following lemma.

\begin{lemma}\label{psi-alphat0}
For any $R>0$, in $\{(x,y): d((x,y),\overline{\mathcal{R}}_0)\le R\}$,
\[y-\frac{1}{\alpha}\psi(0,\alpha x)\rightarrow +\infty \hbox{  uniformly as $\alpha \rightarrow 0^+$},\]
where $\overline{\mathcal{R}}_0$ is the time slice of $\overline{\mathcal{R}}$ at $t=0$.
\end{lemma}

\begin{lemma}\label{u^-}
For any $\alpha>0$ small enough, there exist $\omega>0$ and $\mu>0$ such that for $\delta\in (0,\sigma]$, $u_i^-$ is a subsolution of \eqref{CP} and
\be\label{comparison with subsolution}
u_i^-(t,x,y)\le u(t,x,y), \hbox{ for $t\ge 0$ and $(x,y)\in\R^N$}.
\ee
\end{lemma}

\begin{proof}
To verify the subsolution property, by the definition of $u_i^-$,  we need only to restrict our attention to the domain where  $u_{i, \ast}(t,x,y)=g(\xi_{i,\ast}(t,x,y))+\varepsilon \hat{h}_i(\alpha t,\alpha x)$. By Lemma~\ref{pro-xi-i},
\[\partial_t u_{i,\ast}=g^\prime(\xi_i)\partial_t\xi_i +\varepsilon \alpha \partial_t\hat{h}_i(\alpha t,\alpha x)\ge -c_f g^\prime(\xi_i)-\alpha C\hat{h}_i(\alpha t,\alpha x).\]
As in the proof of Lemma \ref{u+}, if $\alpha$ is small enough, then   $u_{i, \ast}$ satisfies (i)-(ii) of Lemma~\ref{supsubsolution}. By Lemma~\ref{asy-ui}, (iii)' of Lemma~\ref{supsubsolution} is also satisfied by $u_{i, \ast}$. Hence $u_i^-$ is a subsolution of \eqref{CP}.

Next let us check the initial value condition. Since $u(0,x,y)=u_0(x,y)$ satisfies \eqref{initial}, there is an $R>0$ such that 
\[0<g(\xi)\le \delta \hbox{ for $\xi\ge R$} \quad \hbox{ and}  \quad 1-\delta\le g(\xi)<1 \hbox{ for $\xi\le -R$}\]
and
\[|u_0(x,y)-\underline{u}(0,x,y)|\le \delta, \hbox{ for $d((x,y),\mathcal{R}_0)\ge R$}.\]

In $\{x\cdot \nu_i \cos\theta_i +y\sin\theta_i +\tau_i\ge R\}$, 
\[u_{i, \ast}(0,x,y)\le g(x\cdot \nu_i \cos\theta_i +y\sin\theta_i +\tau_i)\le \delta.\]
Thus  
\[u_i^-(0,x,y)=\max\{u_{i, \ast}(0,x,y)-\delta,0\}= 0\le u_0(x,y).\]

In $\{\min_{1\le j\le n}\{x\cdot \nu_j \cos\theta_j +y\sin\theta_j +\tau_j\}\le -R\}$, we have $d((x,y),\mathcal{R}_0)\ge R$. Hence 
\begin{align*}
u_0(x,y)\ge \max\{\underline{u}(0,x,y)-\delta,0\}&\ge \max\{g(x\cdot \nu_i \cos\theta_i +y\sin\theta_i +\tau_i)-\delta,0\}\\
&\ge \max\{u_{i, \ast}(0,x,y)-\delta,0\}=u_i^-(0,x,y).
\end{align*}

Notice that there is $R_1>0$ such that 
\begin{align*}
&\R^N\setminus \left(\{x\cdot \nu_i \cos\theta_i +y\sin\theta_i +\tau_i\ge R\}\cup\{\min_{1\le j\le n}\{x\cdot \nu_j \cos\theta_j +y\sin\theta_j +\tau_j\}\le -R\}\right)\\
\subset& \{d((x,y),\widetilde{P}_{i,0})\le R_1\}.
\end{align*}
Hence there is only one part left, that is, $\{d((x,y),\widetilde{P}_{i,0})\le R_1\}$.

For $(x,y)\in\R^N$ such that $d((x,y),\widetilde{P}_{i,0})\le R_1$ and $d((x,y),\mathcal{R}_0)\ge R_1$, one still has that
\begin{align*}
u_0(x,y)\ge \max\{\underline{u}(0,x,y)-\delta,0\}&\ge \max\{g(x\cdot \nu_i \cos\theta_i +y\sin\theta_i +\tau_i)-\delta,0\}\\
&\ge \max\{u_{i, \ast}(0,x,y)-\delta,0\}=u^-(0,x,y).
\end{align*}
By Lemma~\ref{psi-alphat0}, one can take $\alpha$ small enough such that
\[u_{i, \ast}(0,x,y)\le \delta, \hbox{ for $d((x,y),\overline{\mathcal{R}}_0)\le R_1$}.\]
One can easily show that $\overline{P}_{i,0}=\widetilde{P}_{i,0}$ and hence $\overline{\mathcal{R}}_0$ contains the ridges $\mathcal{R}_{ij,0}$. Then for any $(x,y)$ satisfying  $d((x,y),\widetilde{P}_{i,0})\le R_1$ and $d((x,y),\mathcal{R}_0)\le R_1$, 
\[u_i^-(0,x,y)=\max\{u_{i, \ast}(0,x,y)-\delta,0\}= 0\le u_0(x,y) .\]

In conclusion, $u_i^-(0,x,y)\le u_0(x,y)$ for all $(x,y)\in\R^N$. Then \eqref{comparison with subsolution} follows from the comparison principle.
\end{proof}
\vskip 0.3cm

\subsection{Completion of the proof}
In this subsection, we prove the following proposition. 
Since equation \eqref{RD} is invariant under rotations, the stability of $U_{e_i,\tau_i}$, i.e. Theorem~\ref{Th2}, is a direct consequence of  this proposition.
\begin{proposition}\label{Proposition3}
For the Cauchy problem \eqref{CP}-\eqref{initial}, it holds that
\[\lim_{t\to+\infty}\sup_{(x,y)\in \R^N}|u(t,x,y)-U(t,x,y)|= 0.\]
\end{proposition}

\begin{proof}
By Lemmas~\ref{u+} and \ref{u^-} (with suitably chosen $\varepsilon$, $\alpha$ satisfying the hypothesis in these lemmas), we have 
\[u_i^-(t,x,y)\le u(t,x,y)\le u^+(t,x,y), \hbox{ for every $i\in\{1,\cdots,n\}$, $t\ge 0$ and $(x,y)\in\R^N$}.\]
It implies that for every $i\in\{1,\cdots,n\}$, $t\ge 0$ and $(x,y)\in\R^N$,
\be\label{cp4.1}
u_{i, \ast}(t,x,y)-\|\partial_t\underline{u}_i\|_{\infty}\omega \delta -\delta\le u(t,x,y)\le \overline{u}(t,x,y)+\|\partial_t\overline{u}\|_{\infty}\omega\delta+\delta.
\ee
Thus by taking $\delta>0$ small enough, we can make sure that for every $i\in\{1,\cdots,n\}$, $t\ge 0$ and $(x,y)\in\R^N$,
\[u_{i, \ast}(t,x,y)-\varepsilon\le u(t,x,y)\le \overline{u}(t,x,y)+\varepsilon.\]
 Then by Lemmas~\ref{lem comparison of sup and sub} and \ref{asy-ui}, for any $\varepsilon>0$, there exists an $\rho(\varepsilon)>0$ such that in $\{d((t,x,y),\mathcal{R})\ge \rho(\varepsilon)\}$, 
\be\label{A1-CP}
|u(t,x,y)-\underline{u}(t,x,y)|\le (n^2+1)\varepsilon.
\ee

To complete the proof, assume by the contrary that there exists an $\varepsilon>0$ and a sequence $\{(t_k,x_k,y_k)\}$ such that  as $k\rightarrow +\infty$, $t_k\rightarrow +\infty$ but
\be\label{nconv}
|u(t_k,x_k,y_k)-U(t_k,x_k,y_k)|\geq (n^2+2)\varepsilon.
\ee
By \eqref{A1-CP}, 
\[\limsup_{k\to+\infty}d((t_k,x_k,y_k),\mathcal{R})\leq \rho(\varepsilon).\]
Divide the set $\{1,\cdots,n\}$ into two subsets $I$ and $J$ such that
$x_k\cdot\nu_i\cos\theta_i+y_k\sin\theta_i-c_ft_k+\tau_i$
is uniformly bounded for $i\in I$, while  for $j\in J$,
\be\label{4.2}
\lim_{k\rightarrow +\infty }\left(x_k\cdot\nu_j\cos\theta_j+y_k\sin\theta_i-c_ft_k+\tau_j\right)=+\infty.
\ee
The set $I$ is not empty, but $J$ could be empty.

After passing to a subsequence of $(t_k,x_k,y_k)$ if necessary, there are $\eta_i\in\R$ ($i\in I$) such that for every $i\in I$,
\[\lim_{k\rightarrow +\infty}\left(x_k\cdot\nu_i\cos\theta_i+y_k\sin\theta_i-c_ft_k+\tau_i\right)= \eta_i.\]
Then $\partial\mathcal{P}-(t_k,x_k,y_k)\rightarrow \partial\mathcal{P}_{\infty}$ where $\mathcal{P}_{\infty}$ is the polytope enclosed by hyperplanes
\[\left\{(t,x,y)\in\R\times\R^N; x\cdot\nu_i\cos\theta_i+y\sin\theta_i-c_f t+\eta_i=0\right\}, \ i\in I.\]
By \eqref{4.2}, for every $j\in J$, $g(x\cdot\nu_j\cos\theta_j+y\sin\theta_j-c_f t+\eta_j)\to 0$. Thus $\underline{u}(t+t_k,x+x_k,y+y_k)$ converge to 
\[u_{I, \ast}(t,x,y):=\max_{i\in I}\left\{g(x\cdot\nu_i\cos\theta_i+y\sin\theta_i-c_f t+\eta_i)\right\}.\]
Let $u_k(t,x,y)=u(t+t_k,x+x_k,y+y_k)$. By parabolic estimates and \eqref{A1-CP},  $u_k$ converge (up to a subsequence) to a solution $u_{\infty}$ of \eqref{RDxy}. By Corollary \ref{coro:rhoe} and \eqref{cp4.1} (with $\delta$ sufficiently small),   if $d((t,x,y),\partial\mathcal{P}_{\infty})\ge \rho(\varepsilon)$, then
\[|u_{\infty}(t,x,y)-u_{I, \ast}(t,x,y)|\le (n^2+1)\varepsilon,\]
and if   $d((t,x,y),\mathcal{R}_{\infty})\ge \rho(\varepsilon)$, then
\[
|u_{\infty}(t,x,y)-u_{I, \ast}(t,x,y)|\le \varepsilon.
\]
Since $\varepsilon$ is arbitrary, we get
\begin{align*}
|u_{\infty}(t,x,y)-u_{I, \ast}(t,x,y)|\rightarrow 0, &\hbox{ uniformly as $d((t,x,y),\mathcal{R}_{\infty})\rightarrow +\infty$}.
\end{align*}
By Lemma~\ref{uniqueness},
$u_{\infty}\equiv U_I$ on  $\R\times\R^N$,
where $U_I:=U(\cdot; (e_i,\tau_i)_{i\in I})$. On the other hand, by Proposition \ref{Proposition1} and \eqref{4.2}, $U(t+t_k,x+x_k,y+y_k)$ also converge to $U_I$ as $k\rightarrow +\infty$. This is a contradiction with \eqref{nconv}.
\end{proof}

\section{Characterization of transition fronts}

In this section, we prove Theorem~\ref{Th3}. Throughout this section, $u$ denotes a transition front connecting $0$ and $1$. Recall that we assume there is $T\in\R$ such that for any $t\le T$, there are $n$ vectors $e_i$  (independent of $t$) and constants $\xi_t^i$ ($i=1,2,\cdots,n$) such that the interface $\Gamma_t$ of $u$ is given by 
\be\label{interface}
\Gamma_t=\left\{x\in\R^N;\ \min_{1\le i\le n}\{x\cdot e_i +\xi_t^i\}=0\right\}\neq \emptyset.
\ee
In this section, we use $P(e_i,\xi^i_t)$ to denote the hyperplane $\left\{x\in\R^N; x\cdot e_i+\xi^i_t=0\right\}$.

\begin{lemma}\label{Gtsub}
Assume 
\begin{itemize}
    \item $\Omega_{t_0}^{\pm}=\left\{x\in\R^N; \min_{1\le i\le n}\{x\cdot e_i +\xi^i_{t_0}\}\lessgtr 0\right\}$ for some $t_0\le T$;
    \item there is $e_0\in\mathbb{S}^{N-1}$ such that $e_i\cdot e_0>0$ for all $i\in\{1,\cdots,n\}$.
\end{itemize}
 Then  there is $M>0$ such that for any $t\ge t_0$,
\[\Gamma_t\subset \left\{x\in\R^N; -M\le \min_{1\le i\le n}\{x\cdot e_i+\xi^i_{t_0} -c_f(t-t_0)\}\le M\right\}.\]
\end{lemma}

\begin{proof}
By \cite[Lemma 4]{GH}, there is $a\in (0,1/2)$ such that
\be\label{a}
a\le u(t,x)\le 1-a, \hbox{ for $t\in\R$ and $x\in\Gamma_t$}.
\ee
Take   any $\delta\in (0,\sigma]$ such that $\delta\le a/2$, where $\sigma$ is defined by \eqref{sigma}. By Definition~\ref{def1.1} and the assumption in this lemma, there exists an $M_{\delta}>0$ such that
\begin{eqnarray*}
\left\{\begin{array}{lll}
1-\delta\le u(t_0,x)<1, && \hbox{if} ~  \min_{1\le i\le n}\{x\cdot e_i+\xi^i_{t_0}\}\le -M_{\delta},\\
0< u(t_0,x)\le \delta, && \hbox{if} ~   \min_{1\le i\le n}\{x\cdot e_i+\xi^i_{t_0}\}\ge M_{\delta}.
\end{array}
\right.
\end{eqnarray*}
Since $g(-\infty)=1$ and $g(+\infty)=0$, there exists an $R>0$ such that $0<g(\xi)\le \delta$ for $\xi\ge R$ and $1-\delta\le g(\xi)<1$ for $\xi\le -R$. Then the following four properties hold.
		\begin{enumerate}
			\item[(1)] If $\min_{1\le i\le n}\left\{x\cdot e_i+\xi^i_{t_0}\right\}\ge -M_{\delta}$, then 
   \[g(x\cdot e_i+\xi^i_{t_0}+M_{\delta}+R)\leq\delta.\]
Hence
			\[g(x\cdot e_i+\xi^i_{t_0}+M_{\delta}+R)-\delta\le 0< u(t_0,x), \hbox{ for every $i\in\{1,\cdots,n\}$}.\]
			
			\item[(2)] If $\min_{1\le i\le n}\left\{x\cdot e_i+\xi^i_{t_0}\right\}\le -M_{\delta}$, then $u(t_0,x)\geq 1-\delta$. Hence
			\[g(x\cdot e_i+\xi^i_{t_0}+M_{\delta}+R)-\delta<1-\delta\le  u(t_0,x), \hbox{ for every $i\in\{1,\cdots,n\}$}.\]
			
			\item[(3)] If $\min_{1\le i\le n}\left\{x\cdot e_i+\xi^i_{t_0}\right\}\le M_{\delta}$, then 
   \[   \max_{i=\{1,\cdots,n\}}g(x\cdot e_i +\xi^i_{t_0}-M_{\delta}-R)\ge 1-\delta.\]
		Hence
  \be\label{g>u1}
   \max_{i=\{1,\cdots,n\}}g(x\cdot e_i +\xi^i_{t_0}-M_{\delta}-R)+\delta\ge 1>u(t_0,x).
   \ee
			
			\item[(4)] If $\min_{1\le i\le n}\left\{x\cdot e_i+\xi^i_{t_0}\right\}\ge M_{\delta}$, then $u(t_0,x)\leq \delta$. Hence
			\be\label{g>u2}
   \max_{i=\{1,\cdots,n\}}g(x\cdot e_i +\xi^i_{t_0}-M_{\delta}-R)+\delta>\delta\ge u(t_0,x).\ee
		\end{enumerate} 
		
By Lemma~\ref{supsubsolution}, there exist $\omega>0$ and $\mu>0$ such that for any $i\in\{1,\cdots,n\}$, $t\ge t_0$ and $x\in\R^N$,
\[u(t,x)\ge g\left(x\cdot e_i -c_f(t-t_0)+\xi^i_{t_0}+M_{\delta}+R+\omega\delta e^{-\mu (t-t_0)} -\omega\delta\right)-\delta e^{-\mu (t-t_0)}.\]
Thus for $t\ge t_0$ and $x\in\R^N$,
\[
u(t,x)\ge \max_{i=\{1,\cdots,n\}}g\left(x\cdot e_i -c_f(t-t_0)+\xi^i_{t_0}+M_{\delta}+R+\omega\delta e^{-\mu (t-t_0)} -\omega\delta\right)-\delta e^{-\mu (t-t_0)}.
\]
This together with \eqref{a} implies that if $\min_{1\le i\le n}\left\{x\cdot e_i-c_f(t-t_0)+\xi^i_{t_0}+M_{\delta}+R\right\}\le -R$, then $u(t,x)\ge 1-2\delta\ge 1-a$. In other words, if $x\in\Gamma_t$, then
\be\label{lowb-u}
\min_{1\le i\le n}\left\{x\cdot e_i-c_f(t-t_0)+\xi^i_{t_0}+M_{\delta}+R\right\}\ge -R.
\ee

On the other hand, let $U$ be the entire solution satisfying
\be\label{U>g}
U(t,x)\ge \max_{1\le i\le n}\left\{g\left(x\cdot e_i -c_f (t-t_0)+\xi^i_{t_0}-M_{\delta}-R\right)\right\}, \hbox{ for any $(t,x)\in \R\times\R^N$},
\ee
and 
\[
\left|U(t,x)-\max_{1\le i\le n}\left\{g\left(x\cdot e_i -c_f (t-t_0)+\xi^i_{t_0}-M_{\delta}-R\right)\right\}\right|\rightarrow 0
\]
uniformly as   $d((t,x),\mathcal{R})\rightarrow +\infty$, where $\mathcal{R}$ is the set of all ridges of $\mathcal{P}(P(e_i,c_f (t-t_0)+\xi^i_{t_0}-M_{\delta}-R))$.
 The existence of such $U$ is guaranteed by Theorem~\ref{Th1}. By  \eqref{g>u1}, \eqref{g>u2} and \eqref{U>g},  $u(t_0,x)\le U(t_0,x)+\delta$  for $x\in\R^N$. Then by Lemma~\ref{Ut} and Lemma~\ref{supsubsolution}, there exist $\omega>0$ and $\mu>0$ such that
\be\label{upb-u}
u(t,x)\le U\left(t+\omega\delta e^{-\mu (t-t_0)}-\omega\delta,x\right)+\delta e^{-\mu (t-t_0)}, \hbox{ for $t\ge t_0$ and $x\in\R^N$}.
\ee

Notice that $U$ is a transition front with sets
$\Gamma_t$
and 
$\Omega_t^{\pm}$  as defined in \eqref{definition of Gamma} and \eqref{definition of Omega} respectively.
By   \eqref{upb-u},  if $\min_{1\le i\le n}\{x\cdot e_i-c_f(t-t_0)+\xi^i_{t_0}-M_{\delta}-R\}\ge R$, then $u(t,x) \le 2\delta\le a$. Together with \eqref{a}, this implies that for $x\in\Gamma_t$,
\be\label{upperb-u}
\min_{1\le i\le n}\left\{x\cdot e_i-c_f(t-t_0)+\xi^i_{t_0}-M_{\delta}-R\right\}\le R.
\ee

Combining \eqref{lowb-u} and \eqref{upperb-u}, we finish the proof.
\end{proof}
\vskip 0.3cm

We then show that $e_i$ and $\xi_t^i$ satisfy the following property.

\begin{lemma}
There is $e_0\in\mathbb{S}^{N-1}$ such that $e_i\cdot e_0>0$ for all $i\in \{1,\cdots,n\}$.
\end{lemma}

\begin{proof}
Assume by contrary that there is no such $e_0\in \mathbb{S}^{N-1}$. We have\\
{\bf Claim 1.} There exists a $\nu_1\in\mathbb{S}^{N-1}$ such that $\nu_1\cdot e_i\ge 0$ for all $i\in \{1,\cdots,n\}$.

If this is not true, then the open sets $\{\nu: \nu\cdot e_i<0\}$, $i=1$, $\dots$, $n$, form a covering of $\mathbb{S}^{N-1}$. By continuity, we can find a sufficiently small $\delta>0$ such that $D_i:=\{\nu: \nu\cdot e_i<-\delta\}$, $i=1$, $\dots$, $n$, still form a covering of $\mathbb{S}^{N-1}$. For every $\nu\in D_i$,  if $x=r\nu$  belongs to  the set
\[E_t:=\left\{x\in\R^N;\ \min_{1\le i\le n} \{x\cdot e_i +\xi_t^i\}\ge 0\right\},\]
we must have
\[r\nu\cdot e_i +\xi_t^{i}\ge 0.\]
This implies that
\[0\le r<  \frac{\xi^{i}_t}{-\nu\cdot e_i}\leq \frac{\xi^i_t}{\delta}.\]
Hence $E_t$ 
is bounded. Because this set is either $\Omega_t^+$ or $\Omega_t^-$,  we arrive at a contradiction with the last equation of \eqref{eq1.3}. The proof of Claim 1 is thus complete.

Let $J_1$ be the subset of $\{1,\cdots,n\}$ such that $\nu_1\cdot e_j=0$ for $j\in J_1$, while $\nu_1\cdot e_i>0$ for $i\in\{1,\cdots,n\}\setminus J_1$. If $J_1=\emptyset$, then $\nu_1\cdot e_i>0$ for all $i\in \{1,\cdots,n\}$ and we are done. If $J_1$  has only one element, say $J_1=\{j_1\}$, then one can take $0<\varepsilon<\min_{i\not\in J_1} \nu_1\cdot e_i$ such that $(\nu_1+\varepsilon e_{j_1})\cdot e_i>0$ for every $i\in \{1,\cdots,n\}$. 

In the following we assume $J_1$ has at least two elements.  Then we have\\
{\bf Claim 2.} There exists a $\nu_2\perp \nu_1$ such that $\nu_2\cdot e_j\ge 0$ for all $j\in J_1$. 

If for every $\nu\perp \nu_1$, there exists a $j_{\nu}\in J_1$ such that $\nu\cdot e_{j_{\nu}}=\min_{j\in J_1}\nu\cdot e_j<0$, then by the same proof of Claim 1, we deduce that the projection of $\left\{x\in\R^N;\ \min_{j\in J_1}\{x\cdot e_j +\xi_t^j\}\ge 0\right\}$ onto the vertical plane of $\nu_1$ is bounded.  Because
\[\left\{x\in\R^N;\ \min_{1\le i\le n}\{x\cdot e_i +\xi_t^i\}\ge 0\right\} \subset \left\{x\in\R^N;\ \min_{j\in J_1}\{x\cdot e_j +\xi_t^j\}\ge 0\right\},\]
$\Omega_t^+$ or $\Omega_t^-$ is contained in the cylinder over a bounded set in the vertical plane of $\nu_1$. This is still a contradiction with the definition of $\Omega_t^{\pm}$. 

Let $J_2$ be the subset of $J_1$ such that $\nu_2\cdot e_j=0$ for $j\in J_2\subset J_1$ and $\nu_2\cdot e_j>0$ for $j\in J_1\setminus J_2$. If $J_2=\emptyset$, then $\nu_2\cdot e_j>0$ for all $j\in J_1$ and one can take $0<\varepsilon<\min_{i\not\in J_1} \nu_1\cdot e_i$ such that $(\nu_1+\varepsilon \nu_2)\cdot e_i>0$ for all $i\in\{1,\cdots,n\}$. 

If $J_2$   has only one element, say $J_2=\{j_2\}$, then one can take $0<\varepsilon_1<\min_{i\in J_1} \nu_1\cdot e_i$, $0<\varepsilon_2<\min_{i\in J_1\setminus J_2} \nu_2\cdot e_i$ such that $(\nu_1+\varepsilon_1 (\nu_2 +\varepsilon_2 e_{j_2}))\cdot e_i>0$ for every $i\in \{1,\cdots,n\}$. Thus, $J_2$ has at least two elements. 

Repeating this procedure for $N+1$ times, one gets $\nu_1\perp \nu_2\perp \cdots\perp \nu_N\perp \nu_{N+1}$. Since the spatial dimension is $N$,  this is impossible.
\end{proof}
\vskip 0.3cm

\begin{lemma}
For any $t\le T$, 
\be\label{Ocase2}
\Omega_{t}^{\pm}=\left\{x\in\R^N;\ \min_{1\le i\le n}\{x\cdot e_i+\xi_{t}^i\}\lessgtr 0\right\}.
\ee
\end{lemma}
\begin{proof}
If \eqref{Ocase2} is not true, then
\be\label{Ocase1}
\Omega_{t}^{\pm}=\left\{x\in\R^N;\ \min_{1\le i\le n}\{x\cdot e_i+\xi_{t}^i\}\gtrless 0\right\}
\ee
But by \cite[Theorem 1.5]{W}, $\{u<1/2\}$ has the form $\{t<h(x)\}$, where $h$ is a Lipschitz function on $\R^N$. Moreover, the blowing down limit (up to a subsequence of $\varepsilon\to0$)
\[h_\infty(x):=\lim_{\varepsilon\to 0}\varepsilon h\left(\frac{x}{\varepsilon}\right)\]
is a homogeneous solution of the eikonal equation 
\[|\nabla h_\infty|^2- c_f^{-2}=0.\]
Hence  $h_\infty$ is a convex function. Notice that the blowing down limit of $\{u<1/2\}$ is $\{t<h_\infty(x)\}$. However, \eqref{Ocase1} implies that this blowing down limit (in $\{t\leq 0\}\times\R^N$) is $\cup_t \left\{ \min_{1\le i\le n}x\cdot e_i\leq 0\right\}$, which is not convex. This is a contradiction.
\end{proof}
\vskip 0.3cm

\begin{lemma}
There are $m\in \{1,\cdots,n\}$ vectors $e_i$ satisfying $e_i\neq e_j$ for $i\neq j$ such that for any $t\in\R$, the interfaces $\Gamma_t$ can be written as
\be\label{interface2}
\Gamma_t=\left\{x\in\R^N;\ \min_{1\le i\le m}\{x\cdot e_i-c_f t\}=0\right\},
\ee
and $\Omega_t^{\pm}$ are given by 
\be\label{Omega-pm}
\Omega_t^{\pm}=\left\{x\in\R^N;\ \min_{1\le i\le m}\{x\cdot e_i-c_f t\}\lessgtr 0\right\}.
\ee
\end{lemma}
\begin{proof}
We divide the proof into two steps.

{\bf Step 1. The choice of $m$ vectors. }
If $e_i=e_j$ for some $i\neq j$, then  $d_H(P(e_i,\xi^i_t), P(e_j,\xi^j_t)<+\infty$. By Remark \ref{rmk 2.5}, we can remove one of these two hyperplanes from $\Gamma_t$. So in the following, it is always assumed that $e_i\neq e_j$ for $i\neq j$ and $e_0\cdot e_1=\min_{1\le i\le n} e_0\cdot e_i$. Then  for any $t\le T$, we have
\be\label{P1}
\sup_{x\in P(e_1,\xi^1_t)\cap \Gamma_t}\inf_{2\le i\le n} d\left(x,P(e_i,\xi^i_t)\cap \Gamma_t\right)=+\infty.
\ee
In fact, by taking the sequence $\{x_k:=k(e_0-(e_0\cdot e_1)e_1)\}_{k\in\mathbb{N}}$, we have
\begin{itemize}
    \item $x_k\cdot e_1=0$, which implies that $x_k -\xi^1_t e_1 \in P(e_1,\xi^1_t)\cap \Gamma_t$;
    \item for any $2\le i\le n$,
\[x_k\cdot e_i=k(e_0\cdot e_i -(e_0\cdot e_1)(e_1\cdot e_i))\rightarrow +\infty, \hbox{ as $k\rightarrow +\infty$},\]
which implies that
\[d(x_k-\xi^1_t e_1, P(e_i,\xi^i_t)\cap \Gamma_t)\rightarrow +\infty, \hbox{ as $k\rightarrow +\infty$}.\]
\end{itemize}

By reordering index if necessary, let $m\in\{1,\cdots,n\}$ be the largest integer such that for any $1\le i\le m$,
\[\sup_{t\le T} \sup_{x\in P(e_i,\xi^i_t)\cap \Gamma_t} \min_{j\neq i} \{x\cdot e_j +\xi^i_t\}=+\infty,\]
while this  quantity is bounded for $m+1\le i\le n$. The existence of $m$ is ensured by \eqref{P1}. 

For $t\le T$, define
\[\widetilde{\Gamma}_t:=\left\{x\in\R^N;\ \min_{1\le i\le m}\{x\cdot e_i +\xi^i_t\}=0\right\}.\]
In the following, we will show that
\be\label{finite Hausdorff dist condition}
\sup_{t\le T}d_H\left(\Gamma_t,\widetilde{\Gamma}_t\right)<+\infty.
\ee
Then by Remark \ref{rmk 2.5}, the interface of $u$ can be chosen to be $\widetilde{\Gamma}_t$.

To prove \eqref{finite Hausdorff dist condition}, we will establish
\be\label{finite Hausdorff dist condition 2}
\sup_{t\le T}d_H\left(\widetilde{\Gamma}^\ell_t,\widetilde{\Gamma}^{\ell+1}_t\right)<+\infty, \quad \forall \ell=m,\dots, n-1,
\ee
where
\[\widetilde{\Gamma}_t^l:=\left\{x\in\R^N;\ \min_{1\le i\le l}\{x\cdot e_i +\xi^i_t\}=0\right\}.\]
Notice that $\widetilde{\Gamma}^n_t=\Gamma_t$, $\widetilde{\Gamma}^m_t=\widetilde{\Gamma}_t$.

By the definition, for any $i\geq m+1$,
\be\label{5.15}
\sup_{t\le T} \sup_{x\in P(e_i,\xi^i_t)\cap \Gamma_t} \min_{j\neq i} \left\{x\cdot e_j +\xi^j_t\right\}<+\infty.
\ee
In particular,
\be\label{Gammak-1}
\sup_{t\le T} \sup_{x\in P(e_n,\xi^n_t)\cap \Gamma_t} \min_{1\le j\le n-1} \left\{x\cdot e_j +\xi^j_t\right\}<+\infty.
\ee
This implies that
\be\label{GammatGammak-1} 
\sup_{t\le T} \sup_{x\in\Gamma_t}d(x,\widetilde{\Gamma}_t^{n-1})<+\infty.
\ee

Next we show
\be\label{GammatGammak-1-2}
\sup_{t\le T} \sup_{x\in\widetilde{\Gamma}_t^{n-1}}d(x,\Gamma_t)<+\infty.
\ee
This is equivalent to
\[\sup_{t\le T}\sup_{x\in \widetilde{\Gamma}_t^{n-1}}\left|\min_{1\le i\le n}\left\{x\cdot e_i +\xi^i_{t}\right\}\right|<+\infty.\]
To prove this claim,  assume by the contrary that there is $\{(t_k,x_k)\}_{k\in\mathbb{N}}$ such that
\[x_k \cdot e_i +\xi^i_{t_k}\ge 0, \hbox{ for all $1\le i\le n-1$ and $k\in\mathbb{N}$},\]
and 
\[x_k\cdot e_n +\xi_{t_k}^n\rightarrow -\infty, \hbox{ as $k\rightarrow +\infty$}.\]
The last condition implies that there exists a $\tau_k$ such that $\tau_k\rightarrow +\infty$ as $k\rightarrow +\infty$  and  $(x_k+\tau_k e_0)\cdot e_n +\xi^n_{t_k}=0$. Since $e_i\cdot e_0>0$ for any $i$, $(x_k+\tau_k e_0)\cdot e_i +\xi^i_{t_k}\rightarrow +\infty$ as $k\rightarrow +\infty$. This implies that $x_k+\tau_k e_0\in \Gamma_t$ and $d\left(x_k+\tau_k e_0,\widetilde{\Gamma}_t^{n-1}\right)\rightarrow +\infty$ as $k\rightarrow +\infty$, which is a contradiction with \eqref{GammatGammak-1}. 

Combining \eqref{GammatGammak-1} and \eqref{GammatGammak-1-2}, we get \eqref{finite Hausdorff dist condition 2} in the case of $\ell=n-1$. 

Then we proceed to the proof of \eqref{finite Hausdorff dist condition 2} in the case of $\ell=n-2$. This is similar to the previous case, but we need to show that the condition \eqref{5.15} is preserved after we replace $\Gamma_t$ by $\widetilde{\Gamma}_t^{n-1}$. In other words, 
we claim that for any $m+1\le i\le n-1$,
\be\label{5.16}
\sup_{t\le T} \sup_{x\in P(e_i,\xi^i_t)\cap \widetilde{\Gamma}_t^{n-1}} \min_{1\le j\le n-1,\, j\neq i} \left\{x\cdot e_j +\xi^j_t\right\}<+\infty.
\ee

We prove this by contradiction. So assume for some $m+1\le i\le n-1$, there is a sequence $(t_k,x_k)$ satisfying $x_k\in P(e_i,\xi^i_{t_k})\cap \widetilde{\Gamma}_{t_k}^{n-1}$ and
\be\label{5.17}
\min_{1\le j\le n-1,\, j\neq i} \left\{x_k\cdot e_j +\xi^j_{t_k}\right\}\rightarrow +\infty \quad \hbox{as $k\rightarrow +\infty$}.
\ee
By \eqref{5.15},
\be\label{5.18}
|x_k\cdot e_n +\xi^n_{t_k}|\le M \quad \hbox{for some $M>0$}.
\ee
For $x_k^\prime:=x_k+R(e_n-(e_n\cdot e_i) e_i)$ with any $R>0$, \eqref{5.17} implies that
\[\min_{1\le j\le n-1,\, j\neq i} \left\{x_k^\prime\cdot e_j +\xi^j_{t_k}\right\}\ge R \quad \hbox{ for $k$ large enough},\]
while \eqref{5.18} implies that
\[x_k^\prime\cdot e_n +\xi^n_{t_k}\ge R\left[1-(e_n\cdot e_i)^2\right] -M.\]
These two conditions imply that
\[\min_{1\le j\le n,\, j\neq i} \left\{x_k^\prime\cdot e_j +\xi^j_{t_k}\right\}\ge R\left[1-(e_n\cdot e_i)^2\right] -M \quad \hbox{ for $k$ large enough}.\]
Since $x_k\in P(e_i,\xi^i_{t_k})\cap \widetilde{\Gamma}_{t_k}^{n-1}$, one also has $x_k^\prime\cdot e_i +\xi^i_t=0$. Thus by taking $R\rightarrow +\infty$, one has $x_k^\prime\in P(e_i,\xi^i_{t_k})\cap \Gamma_{t_k}$ and 
\[\min_{1\le j\le n,\, j\neq i} \left\{x_k^\prime\cdot e_j +\xi^j_{t_k}\right\}\rightarrow +\infty,\]
which contradicts \eqref{5.15}.

With \eqref{5.16} in hand, we can proceed as before to get \eqref{finite Hausdorff dist condition 2} with $\ell=n-2$. Then an induction finishes 
 the proof of \eqref{finite Hausdorff dist condition}.

{\bf Step 2.  $\xi_t^i$  can be chosen to be $c_ft$.} Take an arbitrary $i\in \{1,\cdots,m\}$ and let 
$\overline{M}=3M$, where $M$ is defined in Lemma~\ref{Gtsub}. Take $T_1\le T$ so that 
\[\sup_{x\in P(e_i,\xi^i_{T_1})\cap \Gamma_{T_1}}\min_{j\neq i}\left\{x\cdot e_j +\xi^j_{T_1}\right\}\ge \overline{M}.\]
Then there exists an $x_0\in P(e_i,\xi^i_{T_1})\cap\Gamma_{T_1}$ such that
\[\min_{j\neq i}\left\{y\cdot e_j +\xi^j_{T_1}\right\}\ge 0, \hbox{ for all $y\in\overline{B(x_0,\overline{M})}\cap P(e_i,\xi^i_{T_1})\cap \Gamma_{T_1}$}.\]
By Lemma~\ref{Gtsub},  for any $t_0<T_1$,
\be\label{GammaT1}
\Gamma_{T_1}\subset \left\{x\in\R^N;\ -M\le \min_{1\le i\le m}\{x\cdot e_i +\xi^i_{t_0} -c_f (T_1- t_0)\}\le M\right\}.
\ee
Hence 
\[\min_{j\neq i}\left\{y\cdot e_j +\xi^j_{t_0} -c_f (T_1-t_0)\right\}\ge -M, \hbox{ for all $y\in\overline{B(x_0,\overline{M})}\cap P(e_i,\xi^i_{T_1})\cap \Gamma_{T_1}$}.\]
For any $j\neq i$,  take 
\[y=x_0-\overline{M}\frac{e_j-(e_i\cdot e_j)e_i}{|e_j-(e_i\cdot e_j)e_i|}\in\overline{B(x_0,\overline{M})}\cap P(e_i,\xi^i_{T_1})\cap \Gamma_{T_1}.\]
Then
\[x_0\cdot e_j +\xi^j_{t_0} -c_f (T_1- t_0)\ge (x_0-y)\cdot e_j -M\ge \overline{M}\frac{1-(e_i\cdot e_j)^2}{|e_j-(e_i\cdot e_j)\cdot e_i|}\ge 3M.\]
Thus by \eqref{GammaT1}, 
\[\min_{1\le i\le m} \left\{x_0\cdot e_i +\xi^i_{t_0} -c_f (T_1 -t_0)\right\}=x_0\cdot e_i +\xi^i_{t_0} -c_f(T_1-t_0)=-\xi^i_{T_1}+\xi^i_{t_0} -c_f (T_1 -t_0)\in [-M,M].\]
This implies that 
\[\left|\xi^i_{t_0}+c_f t_0\right|<+\infty, \hbox{ for all $t_0<T_1$}.\]
Consequently, for $t_0<T_1$, $\Gamma_{t_0}$ and $\Omega^{\pm}_{t_0}$ can be taken to be 
\[
\Gamma_{t_0}=\left\{x\in\R^N;\ \min_{1\le i\le m}\{x\cdot e_i-c_f t_0\}=0\right\}\hbox{ and }
\Omega_{t_0}^{\pm}=\left\{x\in\R^N;\ \min_{1\le i\le m}\{x\cdot e_i-c_f t_0\}\lessgtr 0\right\}.
\]
Using Lemma~\ref{Gtsub} again, we get \eqref{interface2} and \eqref{Omega-pm}.
\end{proof}
\vskip 0.3cm

\begin{lemma}\label{t0x0}
 For any $(t_0,x_0)\in \R\times\R^N$ satisfying
\[\min_{i=\{1,\cdots,m\}}\left\{x_0\cdot e_i -c_f t_0\right\}\ge 0,\]
it holds that
\[u(t-t_0,x-x_0)\ge u(t,x), \hbox{ for $(t,x)\in\R\times\R^N$}.\]
\end{lemma}
\begin{proof}
The idea of the proof is to compare $u(t-t_0,x-x_0)$ and $u(t,x)$ by the sliding method.
This is similar to the proof  of \cite[Theorem 1.11]{BH2}. We only outline the main strategy. 

Notice that both $u(t-t_0,x-x_0)$ and $u(t,x)$ are transition fronts with the same $\Gamma_t$ and $\Omega_t^\pm$ for any $t\in\R$, where
\[
\Gamma_t=\left\{x\in\R^N;\ \min_{1\le i\le m}\{x\cdot e_i-c_f t\}=0\right\} \hbox{ and } \Omega_t^{\pm}=\left\{x\in\R^N;\ \min_{1\le i\le m}\{x\cdot e_i-c_f t\}\lessgtr 0\right\}.\]

{\bf Step 1.}  $u(t-t_0+s,x-x_0)\ge u(t,x)$ for all $s>0$ large enough. 

This is the same as in the proof of \cite[Theorem 1.11]{BH2}. 

{\bf Step 2.} Then we define
\[s_\ast=\inf\left\{s>0;\ u(t-t_0+s^\prime,x-x_0)\ge u(t,x) \hbox{ in $\R\times\R^N$ for all $s^\prime\ge s$}\right\}.\]
To conclude the lemma, it is sufficient to show that $s_\ast=0$.

Assume by the contrary that $s_\ast>0$.
Two cases may occur: either 
\[\inf_{\{d(x,\Gamma_t)\le M_{\sigma}\}}\left[u(t-t_0+s_\ast,x-x_0)- u(t,x)\right]>0\]
or 
\[\inf_{\{d(x,\Gamma_t)\le M_{\sigma}\}}\left[u(t-t_0+s_\ast,x-x_0)- u(t,x)\right]=0,\] where $M_{\sigma}>0$ is defined by Definition~\ref{def1.1}.

In the first case, because $\partial_tu$ is globally bounded, there is $\eta_0\in (0,s_\ast)$ such that for any $\eta\in [0,\eta_0]$,
\[\inf_{\{d(x,\Gamma_t)\le M_{\sigma}\}}\left[u(t-t_0+s_\ast-\eta,x-x_0)- u(t,x)\right]>0.\]  
By the same proof in Step 1, we get
\[u(t-t_0+s_\ast-\eta,x-x_0)\ge u(t,x), \quad \text{in} ~ \R\times\R^N.\] 
This is a   contradiction with the minimality of $s_\ast$. 

In the second case, there is a sequence $(t_k,x_k)_{k\in\mathbb{N}}$ such that
\[d(x_k,\Gamma_{t_k})\le M_{\sigma} \hbox{ and } u(t_k-t_0+s_\ast,x_k-x_0)- u(t_k,x_k)\rightarrow 0 \hbox{ as $k\rightarrow +\infty$}.\]
Since $u(t-t_0+s_\ast,x-x_0)$ and $u(t,x)$ are solutions of \eqref{RD}, by linear parabolic estimates, we get
\[u(t_k,x_k)- u(t_k+t_0-s_\ast,x_k+x_0)\rightarrow 0 \hbox{ as $k\rightarrow +\infty$}.\]
An induction gives
\[u(t_k,x_k)- u(t_k+ \beta t_0-\beta s_\ast,x_k+\beta x_0)\rightarrow 0 \hbox{ as $k\rightarrow +\infty$ for all $\beta\in \mathbb{N}$}.\]
Then by the fact that
\[(x_k+\beta x_0)\cdot e_i -c_f(t_k+\beta t_0-\beta s_\ast)\rightarrow +\infty, \hbox{ as $\beta\rightarrow +\infty$ for all $i\in\{1,\cdots,m\}$},\]
we deduce that $u(t_k,x_k)\rightarrow 0$ as $k\rightarrow +\infty$. This is impossible since $d(x_k,\Gamma_{t_k})\le M_{\sigma}$ and hence $u(t_k,x_k)\ge \delta>0$ for some $\delta>0$.
\end{proof}
\vskip 0.3cm

Now we are ready to prove Theorem~\ref{Th3}.
\vskip 0.3cm
\begin{proof}[Proof of Theorem~\ref{Th3}]
If $m=1$, then $u$ is an almost planar front. By \cite[Theorem 2.6]{H}, it is an exact planar front. So we assume $m\ge 2$ in the following. 

{\bf Step 1.}
By the sliding method and \eqref{Omega-pm}, one can show that for every $i\in\{1,\cdots,m\}$ and all $\tau$ large enough, 
\[g(x\cdot e_i -c_f t +\tau)\le u(t,x) \hbox{ for $(t,x)\in\R\times\R^N$}.\]
We omit the details here by referring to \cite{BH2} for the sliding method.

For every $i\in \{1,\cdots,m\}$, define
\[\tau_i^\ast =\inf\left\{\tau\in\R;\ g(x\cdot e_i -c_f t +\tau^\prime)\le u(t,x) \hbox{ in $\R\times\R^N$ for all $\tau^\prime\ge \tau$}\right\}.\]
Then we claim that
\be\label{4.1}
\inf_{\{(t,x): d(x,\Gamma_t)\le M_{\sigma}\}} \left[u(t,x)-g(x\cdot e_i -c_f t+\tau_i^\ast)\right]=0.
\ee
In fact, if this is not true, then as in the proof of Lemma~\ref{t0x0}, there is $\eta_0\in (0,\tau_i^\ast]$ such that $g(x\cdot e_i -c_f t+\tau_i^\ast-\eta)\le u(t,x)$ for all $\eta\in [0,\eta_0]$, which contradicts the definition of $\tau_i^\ast$. 

By \eqref{4.1}, there is a sequence $\{(t_{k,1},x_{k,1})\}_{k\in\mathbb{N}}$ satisfying $d(x_{k,1},\Gamma_{t_{k,1}})\le M_{\sigma}$ and
\be\label{xnyn}
\lim_{k\rightarrow +\infty}\left[u(t_{k,1},x_{k,1})-g(x_{k,1}\cdot e_i -c_f t_{k,1}+\tau_i^\ast)\right]= 0.
\ee
The fact that $d(x_{k,1},\Gamma_{t_{k,1}})\le M_{\sigma}$  implies that $\sigma\leq u(t_{k,1},x_{k,1})\leq 1-\sigma$.  This combined with \eqref{xnyn} then implies that  $|x_{k,1}\cdot e_i -c_f t_{k,1}|$ is bounded.

By the definition of $\tau_i^\ast$, we have
\[u(t,x)\geq \max_{1\leq i\leq m} g(x\cdot e_i -c_f t+\tau_i^\ast).\]
If $\min_{j\in\{1,\cdots,m\}\setminus\{i\}} \{x_{k,1}\cdot e_j -c_f t_{k,1}\}$ is also bounded,  this implies that
\[\liminf_{k\rightarrow +\infty}\left[u(t_{k,1},x_{k,1})-g(x_{k,1}\cdot e_i -c_f t_{k,1}+\tau_i^\ast)\right]> 0,\]
which is a contradiction with \eqref{xnyn}.

In conclusion,  $|x_{k,1}\cdot e_i -c_f t_{k,1}|$ is bounded and 
\be\label{As-tkxk}
\min_{j\in\{1,\cdots,m\}\setminus\{i\}} \{x_{k,1}\cdot e_j -c_f t_{k,1}\}\rightarrow +\infty \hbox{ as } k\rightarrow +\infty.
\ee
In the following, we assume without loss of generality that $(t_{k,1},x_{k,1})\in \widetilde{P}(e_i,-c_f t_{k,1})$. This is done by choosing the  nearest point on $\widetilde{P}(e_i,-c_f t_{k,1})$ to $(t_{k,1},x_{k,1})$. Notice that by standard linear parabolic estimates, \eqref{xnyn} still holds. 

{\bf Step 2.} By the sliding method again, one can prove that for $e_i$ and all $\eta_i$ ($i=1,\cdots,m$) negative enough,  
\[U(t,x;e_1,\cdots,e_m,\eta_1,\cdots,\eta_m)\ge u(t,x) \quad \hbox{ in}~ \R\times \R^N.\]
After decreasing $\eta_i$, we may assume without loss of generality that there exists an $r>0$ such that for any $(t,x)\in \R\times\R^N$ satisfying $d((t,x),\widetilde{P}(e_i,-c_f t))\le M_{\sigma}$, we have
\be\label{etai}
g(x\cdot e_i-c_f t+\eta_i)-u(t,x)\ge r.
\ee 
In the following we consider only the sliding procedure in the $e_1$ direction. So $\eta_i$, $i=2$, $\dots$, $m$ will be fixed. Define
\[\eta_1^\ast=\sup\left\{\eta\in\R;\ U(t,x;e_1,\cdots,e_m,\eta^\prime,\cdots,\eta_m) \ge u(t,x) \hbox{ in } \R^N \hbox{ for all $\eta^\prime\le \eta$}\right\},\]
where $U(t,x;\eta^\prime,\cdots,\eta_m)$ stands for $U(t,x;e_1,\cdots,e_m,\eta^\prime,\cdots,\eta_m)$ for short.
Then as in Step 1, we have
\[\inf_{\{(t,x): d(x,\Gamma_t)\le M_{\sigma}\}}\left[U(t,x;\eta_1^\ast,\cdots,\eta_m)- u(t,x)\right]=0.\]
Hence there is a sequence $\{(t_{k,2},x_{k,2})\}_{k\in\mathbb{ N}}$ such that $d(x_{k,2},\Gamma_{t_{k,2}})\le M_{\sigma}$ and
\be\label{eq-t-sequence}
U(t_{k,2},x_{k,2};\eta_1^\ast,\cdots,\eta_m)- u(t_{k,2},x_{k,2})\rightarrow 0 \hbox{ as $k\rightarrow +\infty$}.
\ee
If $\min_{j\in\{2,\cdots,m\}} \{x_{k,2}\cdot e_i -c_f t_{k,2}\}$ is bounded, by \eqref{etai} and \eqref{compareUu}, one can take $x_{k,2}^\prime$ on $\widetilde{P}(e_j,-c_ft_{k,2})$ for some $j\in \{2,\cdots,m\}$ such that $|x_{k,2}-x_{k,2}^\prime|$ is bounded and 
\[U(t_{k,2},x_{k,2};\eta_1^\ast,\cdots,\eta_m)- u(t_{k,2},x_{k,2})\ge \frac{r}{2}>0.\]
 However, standard estimates for linear parabolic equations imply that 
\[U(t_{k,2},x^\prime_{k,2};\eta_1^\prime,\cdots,\eta_m)- u(t_{k,2},x^\prime_{k,2})\rightarrow 0 \hbox{ as } k\rightarrow +\infty.\]
This is a contradiction. 

In conclusion, $|x_{k,2}\cdot e_1 -c_f t_{k,2}|$ is bounded and 
\be\label{txty}
\min_{j\in \{2,\cdots,m\}} \{x_{k,2}\cdot e_j -c_f t_{k,2}\}\rightarrow +\infty \hbox{ as } k\rightarrow +\infty.
\ee
As in Step 1, one can assume without loss of generality that $(t_{k,2},x_{k,2})\in \widetilde{P}(e_1,-c_ft_{k,2})$.

{\bf Step 3.} We need to prove that $\eta_1^\ast=\tau_1^\ast$. By \eqref{AsyES}, we have $\eta_1^\ast\le \tau_1^\ast$. Assume by the contrary that $\eta_1^\ast<\tau_1^\ast$. Since $g$ is a decreasing function, there exists an $\varepsilon$ satisfying
\be\label{epsilon2}
3\varepsilon<g(\eta_1^\ast)-g(\tau_1^\ast)
\ee 
Since $(t_{k,2},x_{k,2})$ satisfies \eqref{eq-t-sequence}, there is $M>0$ large enough such that $(t_{M,2},x_{M,2})\in\widetilde{P}(e_1,-c_f t_{M,2})$ and
\[|U(t_{M,2},x_{M,2};\eta_1^\ast,\cdots,\eta_m)- u(t_{M,2},x_{M,2})|<\varepsilon.\]
Since $U(t,x;\eta_1^\ast,\cdots,\eta_m)$ satisfies \eqref{AsyES} for $e_i$ ($i=1,\cdots,m$) and $\eta_1^\ast,\cdots,\eta_m$, by \eqref{txty}, after increasing $M$, one has 
\[|g(x_{M,2}\cdot e_1 -c_f t_{M,2} +\eta_1^\ast) -u(t_{M,2},x_{M,2})|<2\varepsilon.\]
By \eqref{xnyn}, there is $K>0$ large enough such that $(t_{K,1},x_{K,1})\in\widetilde{P}(e_1,-c_f t_{K,1})$ satisfies
\[\min_{2\le j\le m} \{x_{K,1}\cdot e_j -c_f t_{K,1}\}\ge \max_{2\le j\le m} \{x_{M,2}\cdot e_j -c_f t_{M,2}\},\]
and 
\[|g(x_{K,1}\cdot e_1 -c_f t_{K,1} +\tau_1^\ast) -u(t_{K,1},x_{K,1})|<\varepsilon.\]
Let $(t_0,x_0):=(t_{K,1}-t_{M,2},x_{K,1}-x_{M,2})$. It belongs to $\widetilde{P}(e_1,-c_f t_0)$ and satisfies $\min_{2\le j\le m} \{x_0\cdot e_j -c_f t_0\}\ge 0$.
Then we have
\begin{align*}
g(\tau_1^\ast)&=g(x_{K,1}\cdot e_1-c_f t_{K,1} +\tau_1^\ast)\\
&
\ge u(t_{K,1},x_{K,1})-\varepsilon\\
&=u(t_{M,2}-t_0,x_{M,2}-x_0)-\varepsilon \\
&\ge u(t_{M,2},x_{M,2})-\varepsilon \quad\quad \text{(by Lemma~\ref{t0x0})}\\
&\ge g(x_{M,2}\cdot e_1-c_f t_{M,2} +\eta_1^\ast) -3\varepsilon\\
&= g(\eta_1^\ast)-3\varepsilon.
\end{align*}
This is a contradiction with \eqref{epsilon2}. Therefore, $\eta_1^\ast\equiv \tau_1^\ast$.

Repeating the above argument for $i=2$, $\dots$, $m$, one can define $\eta_i^\ast$ and prove that $\eta_i^\ast\equiv \tau_i^\ast$ for all $i=\{1,\cdots,m\}$. In other words,
\[\max_{1\le i\le n} \{g(x\cdot e_i-c_f t+\tau_i^\ast)\}\le u(t,x)\le U(t,x;\eta_1,\cdots,\tau_i^\ast,\cdots,\eta_m), \hbox{ for every $i\in \{1,\cdots,m\}$}.\]
Thus $u$ satisfies \eqref{AsyES} for $\{e_1,\cdots, e_m, \tau_1^\ast,\cdots,\tau_m^\ast\}$. By the uniqueness result in Proposition \ref{uniqueness}, $u(t,x)\equiv U(t,x;\tau_1^\ast,\cdots,\tau_m^\ast)$.
\end{proof}
\vskip 0.3cm

\noindent
\textbf{Acknowledgement.} H. Guo would like to thank Prof. Xinfu Chen for many discussions. H. Guo was supported by the fundamental research funds for the central universities and the National Natural Science Foundation of China under Grant 12101456. K. Wang was supported by the National Natural Science Foundation of China under Grant 12221001.

\end{document}